\documentclass[12pt]{article}

\usepackage{amsfonts, amsmath, amscd}
\usepackage[psamsfonts]{amssymb}

\usepackage[usenames]{color}

\usepackage{amssymb}

\newtheorem{theorem}{Theorem}
\newtheorem{lemma}[theorem]{Lemma}
\newtheorem{pro}[theorem]{Proposition}
\newtheorem{defin}[theorem]{Definition}
\newtheorem{cor}[theorem]{Corollary}
\newtheorem{problem}[theorem]{Problem}

\newenvironment{proof}{\noindent {\sl Proof.}}{$\Box$ \bigskip}

\newtheorem{rmk}[theorem]{Remark}
\newtheorem{exa}[theorem]{Example}

\numberwithin{equation}{section}
\numberwithin{theorem}{section}

\begin{document}


\title{Topologies on groups determined by sets of convergent sequences}

\author{S.~S.~Gabriyelyan
\\{\footnotesize Department of Mathematics,}
\\{\footnotesize Ben-Gurion University of the Negev,}
\\{\footnotesize  Beer-Sheva, P.O. 653,}
\\{\footnotesize  Israel}
\\{\footnotesize  e-mail: saak@math.bgu.ac.il}}

\date{}


\maketitle


\begin{abstract}
A Hausdorff  topological group $(G,\tau)$ is called an $s$-group and $\tau$ is called an $s$-topology if there is a set $S$ of sequences in $G$ such that $\tau$ is the finest Hausdorff group topology on $G$ in which every sequence of $S$ converges to the unit. The class $\mathbf{S}$ of all  $s$-groups contains all sequential Hausdorff groups and it is finitely multiplicative. A quotient group of an $s$-group is an $s$-group. For a non-discrete topological group $(G,\tau)$ the following three assertions are equivalent: 1) $(G,\tau)$ is an $s$-group, 2) $(G,\tau)$ is a quotient group of a Graev free topological group over a metrizable space, 3) $(G,\tau)$ is a quotient group of a Graev free  topological group over a sequential Tychonoff space. The Abelian version of this characterization of $s$-groups holds as well.
\end{abstract}





\section{Introduction}

{\bf I. Notations and preliminaries result.} A group $G$ with the discrete topology is denoted by $G_{d}$. The unit of $G$ is denoted by $e_G$. The subgroup generated by a subset $A$ of $G$ is denoted by $\langle A\rangle$. The group of all permutations of the set $\{ 0, 1,\dots,n-1\}$  is denoted by $\mathbb{S}_n$. Set $\omega =\mathbb{N} \cup \{ 0\}$. The filter of all open neighborhoods of $e_G$ of a topological group $(G,\tau)$ is denoted by $\mathcal{U}_G$.

Let $\mathbf{u}=\{ u_n\}$ be a non-trivial sequence in a  group $G$. The following  very important question has been studied by many authors as  Graev \cite{Gra}, Nienhuys \cite{Nie}, and others:
\begin{problem} \label{prob1}
Is there a Hausdorff group topology $\tau$ on $G$ such that $u_n \to e_G$ in $(G,\tau)$?
\end{problem}
Protasov and  Zelenyuk \cite{ZP1, ZP2} obtained a criterion that gives the complete answer to this question \cite[Theorems 2.1.3 and 3.1.4]{ZP2} (see also the cases (ii) and (iii) in Section \ref{sec1}). Following  \cite{ZP1}, we say that a sequence $\mathbf{u} =\{ u_n \}$ in a group $G$ is a $T$-{\it sequence} if there is a Hausdorff group topology on $G$ in which $u_n $ converges to $e_G$. The group $G$ equipped with the finest Hausdorff group topology $\tau_\mathbf{u}$ with this property is denoted by $(G, \tau_\mathbf{u})$.  A $T$-sequence $\mathbf{u}=\{ u_n\}$ is called {\it trivial} if there is $n_0$ such that $u_n =e_G$ for every $n\geq n_0$. Evidently, if $\mathbf{u}$ is trivial, then $\tau_\mathbf{u}$ is discrete. A sequence $\mathbf{u}=\{ u_n \}_{n\in \omega}$ in  $G$ is called {\it one-to-one} if  $u_n \not= u_m$ for all $n,m\in\omega$ such that $n\not= m$.

One of the most important notions in the article is the notion of a {\it Graev free (Abelian) topological group} introduced in \cite{Gra}.
\begin{defin} {\rm \cite{Gra}}
Let $X$ be a Tychonoff space with a fixed point (basepoint) $e\in X$. A topological group $F(X)$  is called {\em  the Graev free topological  group} over  $X$ if $F(X)$  satisfies the following conditions:
\begin{enumerate}
\item[{\rm (i)}] There is a continuous mapping $i: X\to F(X)$  such that $i(e)=e_{F(X)}$  and $i(X)$ algebraically generates $F(X)$.
\item[{\rm (ii)}] If a continuous mapping $f: X\to G$ to a topological  group $G$ satisfies $f(e)=e_G$, then there exists a continuous homomorphism ${\bar f}: F(X) \to G$  such that $f={\bar f} \circ i$.
\end{enumerate}
\end{defin}
The Graev free {\it Abelian} topological group $A(X)$ over a Tychonoff space $X$ with a fixed point is defined similarly. Note that  $A(X)$ is a quotient group of the Graev free topological group $F(X)$. Let us note also that the mapping $i: X\to F(X)$ (respectively $i: X\to A(X)$) is a topological embedding and $i(X)$ is closed in $F(X)$ (respectively in $A(X)$) \cite{Gra}.

The following group is the simplest and, as it turns out, the most important example of an Abelian topological group with the topology generated by a $T$-sequence (see Theorem \ref{t02} below). For simplicity's sake, let us denote  by $\mathbb{Z}^{\mathbb{N}}_0$ the direct sum $\bigoplus_\mathbb{N} \mathbb{Z} \subset \mathbb{Z}^\mathbb{N}$. Then the sequence $\mathbf{e} =\{ e_n\} \in\mathbb{Z}_0^{\mathbb{N}}$, where $e_1 =(1,0,0,\dots), e_2 = (0,1,0,\dots), \dots$, converges to zero in the topology induced on $\mathbb{Z}^{\mathbb{N}}_0$ by the product topology on $(\mathbb{Z}_d)^\mathbb{N}$. Thus $\mathbf{e}$ is a $T$-sequence. The topology $\tau_\mathbf{e}$ and the dual group of $(\mathbb{Z}_0^{\mathbb{N}} , \tau_\mathbf{e})$ are described in \cite{Gab} explicitly.  The convergent sequence $\mathbf{e}$ with zero forms a compact metrizable space $E:=\mathbf{e}\cup \{ 0\}$. Denote by $F(\mathbf{e})$ (respectively $A(\mathbf{e})$) the Graev free topological (respectively Abelian) group generated by $E$  ($0$ is basepoint of $E$). By the definitions of $F(\mathbf{e})$, $A(\mathbf{e})$  and $\tau_\mathbf{e}$, we obtain that
\[
F(\mathbf{e})\cong (F,\tau_\mathbf{e}) \; \mbox{ and } \; A(\mathbf{e}) \cong (\mathbb{Z}_0^{\mathbb{N}} , \tau_\mathbf{e}),
\]
where $F$ is the free group over the alphabet $\mathbf{e} =\{ e_1, e_2, \dots \}$.

A subset $A$ of a topological space $\Omega$ is called {\it sequentially open} if whenever a sequence $\{ u_n\} $ converges to a point of $A$, then  all but finitely many of the members $u_n$ are contained in $A$. The space $\Omega$ is called {\it sequential} if any subset $A$ is open if and only if $A$ is sequentially open. Every sequential Hausdorff space is a $k$-space \cite[Theorem 3.3.20]{Eng}. Recall that a topological space $X$ is called a $k$-{\it space} if $X$ is Hausdorff and a subset $A$ of $X$ is closed in $X$ if and only if $A\cap K$ is compact for every compact subset $K$ of $X$. Let $\Omega $ be a sequential space. For a subset $E$ of $\Omega$ the set of all limit points of all convergent sequences in $E$ is denoted by $[E]^{s}$. Clearly, $[E]^{s} \subseteq \mathrm{cl}(E)$. Set $[E]_0 =E$, $[E]_{\alpha +1} = \left[ [E]_{\alpha }\right]^{s}$ and $[E]_{\alpha } =\cup_{\beta <\alpha} [E]_{\beta}$ for a limit ordinal $\alpha$.
The {\it sequential order} ${\rm so } (\Omega)$  of  $\Omega$ is the least ordinal $\alpha$ such that $[E]_{\alpha } ={\rm cl } (E)$ for every subset $E\subseteq \Omega$. Note that  $\mathrm{so}(\Omega)\leq \omega_1$, where $\omega_1$ is the first uncountable ordinal \cite[Proposition 1.3]{ArF}.
The space $\Omega$ is a {\it Fr\'{e}chet-Urysohn} space if for any point $x$ in the closure of an arbitrary subset $E$ there is a sequence in $E$ converging to $x$. So $\Omega$ is  Fr\'{e}chet-Urysohn if and only if ${\rm so } (\Omega)=1$. For more about sequential and Fr\'{e}chet-Urysohn spaces see \cite{Eng, Fra}. Franklin \cite{Fra} gave  the following characterization of sequential spaces:
\begin{theorem} \label{t01} {\rm \cite{Fra}}
A topological space is sequential if and only if it is a quotient of a metric space.
\end{theorem}
Now we formulate the most important properties of the topology $\tau_\mathbf{u}$ generated by a $T$-sequence $\mathbf{u}$ in a  group $G$ (see Theorems 2.3.1 and 2.3.10, Corollary 4.1.5 and Exercise 4.3.1 of \cite{ZP2}). Recall that a topological space $X$ is called {\it hemicompact} if there exists an increasing sequence $\{ K_n \}_{n\in\omega}$ of compact subsets of $X$ such that every compact subset $K$ of $X$ is contained in $K_n$ for some $n\in\omega$ (see \cite[3.4.E]{Eng}).
\begin{theorem} \label{p01} {\rm \cite{ZP2}}
Let  $\mathbf{u}$ be a non-trivial $T$-sequence in a  group $G$.
Then  $(G,\tau_\mathbf{u})$ is a sequential space of sequential order $\omega_1$. Moreover, if $G$ is countable, then   $(G,\tau_\mathbf{u})$ is hemicompact.
\end{theorem}

For other non-trivial examples of sequential Hausdorff Abelian groups see \cite{ChMPT}.

Let $X$ and $Y$ be topological spaces. Following Siwiec \cite{Siw}, a continuous mapping $f: X\to Y$ is called {\it sequence-covering} if whenever $\{ y_n \}_{n\in\omega}$ is a sequence in $Y$ converging to a point $y\in Y$, there exists a sequence of points $ x_n \in f^{-1}(y_n)$ for $n\geq 1$ and $x\in f^{-1} (y)$ such that $x_n \to x$. It is clear that any sequence-covering mapping must be surjective. If $X$ and $Y$ are topological groups, a continuous homomorphism $p: X\to Y$ is sequence-covering if and only if it is surjective and for every sequence $\{ y_n \}$ converging to the unit $e_Y$ there is a sequence $\{ x_n \}$ converging to $e_X$ such that $p(x_n) = y_n$. A mapping $f: X\to Y$ is called {\it sequentially continuous} if $x_n\to x_0$ in $X$ implies that $f(x_n) \to f(x_0)$ in $Y$. Sequentially continuous mappings on products of topological spaces were considered by Mazur \cite{Maz}. The following important notion was introduced by Noble \cite{Nob}:
\begin{defin} {\rm \cite{Nob}} \label{ds}
A Hausdorff topological group $(G,\tau)$ is called an {\em $s$-group} if each sequentially continuous homomorphism from $(G,\tau)$ to a Hausdorff topological group is continuous.
\end{defin}
Let us note that this definition gives an {\it external} characterization of $s$-groups. Following \cite{Nob}, we
denote by $\mathfrak{s}$ the first cardinal such that there exists a noncontinuous, sequentially continuous
mapping $f : 2^\mathfrak{s} \to \mathbb{R}$. Mazur proved that $\mathfrak{s}$ is weakly inaccessible (i.e., $\mathfrak{s}$ is uncountable, regular, and a strong limit cardinal) \cite{Maz}. In particular, $\mathfrak{s} > \aleph_0$. The following theorem is proved by Noble \cite[Theorem 5.4]{Nob}, who changed the Mazur's factorization method by a stronger one:
\begin{theorem} {\rm \cite{Nob}} \label{Nob}
Every sequentially continuous homomorphism from a product of
less than $\mathfrak{s}$ $s$-groups into a Hausdorff group is continuous, i.e., the product is an $s$-group.
\end{theorem}
Hu\v{s}ek \cite{Hus} reproved  Noble's theorem \ref{Nob}. Other results and historical  remarks about $s$-groups can be found in \cite{AJP, Hus, Sha}.

{\bf II. Main results.} This article was inspired by the following two arguments. Firstly, it is natural to know the answer to the following  generalization of Problem \ref{prob1}:
\begin{problem} \label{prob2}
Let $G$ be a group and $S$ be a set of sequences in $G$. Is there a Hausdorff group topology $\tau$ on $G$ in which every sequence of $S$ converges to the unit $e_G$?
\end{problem}
Secondly, it is interesting to understand for which class of topological groups including all sequential groups we can obtain an analogue of the Franklin theorem \ref{t01}. Of course,  we cannot obtain all sequential groups as Hausdorff quotients of  metric ones.

By analogy with $T$-sequences we define:
\begin{defin} \label{d01}
{\it Let $G$ be a group and $S$ be a set of sequences in $G$. The set $S$ is called a {\em $T_s$-set of sequences} if there is a Hausdorff  group topology on $G$ in which all sequences of $S$ converge to $e_G$. The finest Hausdorff  group topology with this property is denoted by $\tau_S$. }
\end{defin}

The set of all $T_s$-sets  of sequences in a group $G$ we denote by $\mathcal{TS} (G)$.
It is clear that, if $S\in \mathcal{TS}(G)$, then $S' \in \mathcal{TS}(G)$  for every nonempty subset $S'$ of $S$ and every sequence $\mathbf{u}\in S$ is a $T$-sequence. Evidently, $\tau_S \subseteq \tau_{S'}$. Also, if $S$ contains only trivial sequences, then $S\in \mathcal{TS}(G)$ and $\tau_S$ is discrete.
By definition, $\tau_\mathbf{u}$ is finer than $\tau_S$ for every $\mathbf{u} \in S$. Thus, if $U$ is open in $\tau_S$, then it is open  in $\tau_\mathbf{u}$ for every $\mathbf{u}\in S$.
So, by definition, we obtain that $\tau_S \subseteq \bigwedge_{\mathbf{u}\in S} \tau_\mathbf{u}$, where $\bigwedge_{\mathbf{u}\in S} \tau_\mathbf{u}$ denotes the {\it intersection} of the topologies
$\tau_\mathbf{u}$ (i.e., $U$ is open in $\bigwedge_{\mathbf{u}\in S} \tau_\mathbf{u}$ if and only if $U\in \tau_\mathbf{u}$ for every $\mathbf{u}\in S$). Denote  by  $\inf \tau_\mathbf{u}$ the {\it infimum} of the topologies $\tau_\mathbf{u}$ in the lattice of all group topologies.
\begin{theorem} \label{char1}
A non-empty family $S$ of $T$-sequences on a group $G$ is a $T_s$-set if and only if
$\inf_{\mathbf{u}\in S} \tau_\mathbf{u}$ is Hausdorff. In such a case, $\tau_S = \inf_{\mathbf{u}\in S} \tau_\mathbf{u}$.
\end{theorem}

Let us note that Definition \ref{ds} of $s$-groups differs from the following one. Nevertheless, we shall show in Theorem \ref{charac} that both definitions define the same class of topological groups.
\begin{defin} \label{d02}
{\it A Hausdorff  topological group $(G, \tau)$ is called an {\em $s$-group} and the topology $\tau$ is called an {\em  $s$-topology} on $G$ if there is $S\in \mathcal{TS}(G)$ such that $\tau =\tau_S$. }
\end{defin}
In other words, $G$ is an $s$-group if and only if its topology is completely determined by convergent sequences.

The family of all $s$-groups (respectively all Abelian $s$-groups) is denoted by $\mathbf{S}$ (respectively by $\mathbf{SA})$. Evidently, if $H$ is an open subgroup of a topological group $G$ and if $H$ is an $s$-group, then $G$ is an $s$-group as well.

We shall show that our definition of $s$-groups has essential advantages than Noble's one. Firstly, Definition \ref{d02} is {\it internal} and does not use any external objects as homomorphisms
into other topological groups. Secondly, this internal characterization allows us to describe  a base $\mathcal{B}_G$ of $\mathcal{U}_G$ of an $s$-group $(G,\tau)$ and to solve
Problem \ref{prob2} (Section \ref{sec1}). The description of   $\mathcal{B}_G$ has been used repeatedly throughout the article. For instance,  we prove the equivalence of Definitions
\ref{ds} and \ref{d02} with the help of this description (Section \ref{sec2}). Thirdly, making use of the topology of $s$-groups important properties of $s$-groups are established
(Sections \ref{sec3}-\ref{sec4}). In particular, it is shown that every sequential group is an $s$-group. At last, we characterize $s$-groups as quotients of the Graev free
topological group over metrizable spaces (Section \ref{sec5}). This result can be viewed as a natural analogue of Franklin's theorem \ref{t01}. To prove these results is the mail goal of the article.

One of the most natural way to find $T_s$-sets of sequences is as follows. Let $(G, \tau)$ be a Hausdorff topological group. We denote the set of all  sequences of $(G, \tau)$ converging to the unit  by $S(G,\tau)$; that is,
\[
S(G,\tau) =\left\{ \mathbf{u} =\{ u_n \} \subset G : \; u_n \to e_G \mbox{ in } \tau \right\}.
\]
It is clear that $S(G,\tau)\in \mathcal{TS}(G)$ and $\tau \subseteq \tau_{S(G,\tau)}$.
Below (see Proposition \ref{p02}) we shall show that the set $S(G,\tau)$ defines the $s$-topology $\tau$.

The article is organized as follows.
In Section \ref{sec1}, Problem \ref{prob2} is solved. More precisely, we give a criterion for a subset $S$ of sequences in a group $G$ to be a $T_s$-set of sequences essentially generalizing Theorem 2.1.3 of \cite{ZP2}.

The next theorem is the main result of Section \ref{sec2}:
\begin{theorem} \label{charac}
A topological group $(G,\tau)$ is an $s$-group if and only if every sequentially continuous homomorphism from $(G,\tau)$ to a Hausdorff topological group $(X,\sigma)$ is continuous.
\end{theorem}
So Theorem \ref{charac} shows  that  Definitions \ref{ds} and \ref{d02} are equivalent.

In Section \ref{sec3} we consider some basic properties of $s$-groups. We show that
the class $\mathbf{S}$ closed under taking  quotients and is finitely multiplicative:
\begin{theorem} \label{t03}
{\it Let $S$ be a $T_s$-set of sequences in a group $G$, $H$ be a closed normal subgroup of $(G, \tau_S)$ and let $\pi$ be the natural projection from $G$ onto  the quotient group $G/H$. Then $\pi(S)$ is a $T_s$-set of sequences in $G/H$ and $G/H \cong (G/H, \tau_{\pi(S)})$.}
\end{theorem}

\begin{theorem} \label{t04}
{\it Let $(G, \tau)$ and $(H,\nu)$ be Hausdorff  topological groups. The following conditions are equivalent:
\begin{enumerate}
\item[{\rm (i)}] $(G, \tau)$ and $(H,\nu)$ are $s$-groups;
\item[{\rm (ii)}] The direct product $(G, \tau)\times (H,\nu)$ is an $s$-group.
\end{enumerate} }
\end{theorem}
In spite of Noble's theorem \ref{Nob} is essentially stronger than Theorem \ref{t04}, we give the proof  of Theorem \ref{t04} to demonstrate our (sequential) approach to $s$-groups.


The next theorem shows that the class $\mathbf{S}$  contains the class $\mathbf{Seq}$ of all sequential Hausdorff groups:
\begin{theorem} \label{t05}
{\it If $(G, \tau)$ is a sequential Hausdorff group, then it is an $s$-group. More precisely, $\tau = \tau_{S(G,\tau)} =\bigwedge_{\mathbf{u}\in S(G,\tau)} \tau_\mathbf{u}$.}
\end{theorem}
It is well known that a closed subgroup of a sequential group is sequential as well. The example in \cite[Theorem 6]{BaT} (see Example \ref{exa1} below) shows that $s$-groups may contain  closed subgroups which are neither  $s$-groups  nor   $k$-spaces. So the class $\mathbf{S}$ contains {\it properly} the class $\mathbf{Seq}$. This example shows also that the class $\mathbf{Seq}$ is not closed under taking finite products, contrarily to the class $\mathbf{S}$ (Theorem \ref{t04}).

Theorem \ref{t05} justifies the following:
\begin{defin}
{\it Let $(G,\tau)$ be a  Hausdorff  topological group. The group $(G, \tau_{S(G,\tau)})$ is called $\mathbf{s}$-{\em refinement} of $(G,\tau)$ and we denote it by $\mathbf{s}(G,\tau)$.}
\end{defin}
So, $(G,\tau)$ is an $s$-group if and only if $\mathbf{s}(G,\tau)=(G,\tau)$.

Let $(G, \tau)$ be an $s$-group. Then, by definition, there is $S\in \mathcal{TS}(G)$ such that $\tau =\tau_S$. It is natural to consider the following cardinal number:
\begin{defin}
{\it Let $(G,\tau)$ be an $s$-group. Set}
\[
r_s (G, \tau)= \min\left\{ |S| : \,  S \in \mathcal{TS} (G) \mbox{ and } \tau_S = \tau \right\}.
\]
\end{defin}
In other words, $r_s (G, \tau)$ is the minimal possible cardinality of those $T_s$-sets of sequences in the group $G$ which generate the $s$-topology $\tau$. In Section \ref{sec4} we prove the following:
\begin{theorem} \label{t001}
Let $(G, \tau)$ be a non-discrete $s$-group. If $r_s (G, \tau)< \omega_1$, then $(G, \tau)$ is sequential and ${\rm so } (G, \tau)=\omega_1$.
\end{theorem}
Let $(G,\tau)$ be a non-discrete Fr\'{e}chet-Urysohn group. Then $\mathrm{so}(G,\tau)=1$ and $(G,\tau)$ is an $s$-group by Theorem \ref{t05}. Thus Theorem \ref{t001} immediately yields:
\begin{cor} \label{c01}
If $(G, \tau)$ is a non-discrete Fr\'{e}chet-Urysohn (in particular, metrizable) group, then $r_s (G, \tau)\geq \omega_1$.
\end{cor}
This means that the topology of a Fr\'{e}chet-Urysohn group cannot be described by a countable set of sequences converging to the unit.

Nyikos in \cite[Problem 4]{Nyi} asked whether the sequential order of a sequential topological group is $\omega_1$ if the group is not Fr\'{e}chet-Urysohn. Under the Continuum Hypothesis Shibakov \cite{Sh1, Sh2} gave a consistent negative answer to Nyikos' question by constructing a sequential topological group of any sequential order. Nevertheless, without any additional set-theoretic assumption beyond ZFC, the answer to Nyikos' question is still open. The difficulty of this question is explained by Theorem \ref{t001}: the topology of an $s$-group $(G,\tau)$ with $1< \mathrm{so}(G,\tau)<\omega_1$ cannot be determined by a countable $T_s$-set of sequences (that implies, in particular, major technical difficulties of constructing such sequential groups, see \cite{Sh1, Sh2}).

In Section \ref{sec5}  a characterization of $s$-groups is given. It turns out that we can describe $s$-groups as  quotients of   Graev free topological groups over  sequential  Tychonoff  spaces.  The following  theorem  is the main in the article:
\begin{theorem} \label{t06}
Let $(G,\tau)$ be a non-discrete Hausdorff (respectively Hausdorff Abelian) topological group. The following assertions are equivalent:
\begin{enumerate}
\item[{\rm (i)}]  $(G,\tau)$ is an $s$-group;
\item[{\rm (ii)}]  $(G,\tau)$ is a quotient of a Graev free (respectively Graev free  Abelian) topological group over a metrizable  space;
\item[{\rm (iii)}]  $(G,\tau)$ is a quotient of a Graev free (respectively  Graev free Abelian) topological group over a Fr\'{e}chet-Urysohn Tychonoff space;
\item[{\rm (iv)}]  $(G,\tau)$ is a quotient of a Graev free (respectively Graev free  Abelian) topological group over a sequential  Tychonoff  space.
\end{enumerate}
\end{theorem}
In particular, assume that a metrizable space $X$ is such that the set of all its non-isolated points is not separable. Then the tightness  of the Graev free Abelian topological group over $X$ is uncountable \cite{AOP}. Hence there are $s$-groups whose tightness is uncountable.

The natural analogue of Problem \ref{prob1} for {\it precompact} group topologies on $\mathbb{Z}$ is studied by Raczkowski \cite{Rac}. Following \cite{BDM} and motivated by \cite{Rac}, we say that a sequence $\mathbf{u} =\{ u_n\}$ is a $TB$-{\it sequence} in a group $G$ if there is a {\it precompact} Hausdorff group topology on $G$ in which  $u_n \to e_G$.  The group $G$ equipped with the finest precompact Hausdorff group topology $\tau_{b\mathbf{u}}$ with this property is denoted by $(G, \tau_{b\mathbf{u}})$.

The counterparts of Problem \ref{prob2} and Definitions \ref{d01} and \ref{d02} for precompact group topologies are defined as follows:
\begin{problem} \label{prob22}
Let $G$ be a group and $S$ be a set of sequences in $G$. Is there a precompact Hausdorff group topology $\tau$ on $G$ in which every sequence of $S$ converges to the unit $e_G$?
\end{problem}

\begin{defin} \label{d012}
{\it Let $G$ be a group and $S$ be a set of sequences in $G$. The set $S$ is called a {\em $T_{bs}$-set of sequences} if there is a precompact Hausdorff group topology on $G$ in which all sequences of $S$ converge to $e_G$. The finest  precompact Hausdorff group topology with this property is denoted by $\tau_{bS}$. }
\end{defin}

The set of all  $T_{bs}$-sets of sequences in a group $G$ we denote by  $\mathcal{TBS} (G)$.
Clearly, if  $S\in \mathcal{TBS}(G)$, then  $S' \in \mathcal{TBS}(G)$ for every nonempty subset $S'$ of $S$, $\tau_{bS} \subseteq \tau_{bS'}$ and every sequence $\mathbf{u}\in S$ is a $TB$-sequence.
Evidently, if $S\in \mathcal{TBS}(G)$, then $S\in \mathcal{TS}(G)$ as well. But, in general, the converse is not true even for one $T$-sequence. For instance, there is a $T$-sequence $\mathbf{v}$ on $\mathbb{Z}$ such that the group $(\mathbb{Z}, \tau_\mathbf{v})$ has no non-trivial characters \cite{ZP1}, but every precompact group has sufficiently many continuous characters.

\begin{defin} \label{d022}
{\it A Hausdorff  topological group $(G, \tau)$ is called a  {\em $bs$-group} and the topology $\tau$ is called  a  {\em $bs$-topology} on $G$ if there is  $S\in \mathcal{TBS}(G)$ such that $\tau =\tau_S$. }
\end{defin}
In the article we deal with $T_s$-sets of sequences and $s$-topologies only.  The class $\mathbf{BS}$ of all  $bs$-groups and its connection with the class $\mathbf{S}$ will be considered in forthcoming articles.

In the last section we pose some open questions.


\section{A criterion to be a $T_s$-set of sequences} \label{sec1}

First of all we prove Theorem \ref{char1}:

{\it Proof of Theorem} \ref{char1}.
Let us prove that $\inf_{\mathbf{u}\in S} \tau_\mathbf{u}$ is the finest group topology on $G$ in which all sequences from $S$ converge to $e_G$. Indeed, since $u_n \to e_G$ in $\tau_\mathbf{u}$ for every $\mathbf{u} \in S$, obviously $u_n \to e_G$ in $\inf_{\mathbf{u}\in S} \tau_\mathbf{u}$. Conversely, if all $\mathbf{u} \in S$ converge to $e_G$ in some group topology $\tau$ on $G$, then
$\tau \leq \tau_\mathbf{u}$ for every $\mathbf{u} \in S$ by the definition of $\tau_\mathbf{u}$, and hence $\tau \leq \inf_{\mathbf{u}\in S} \tau_\mathbf{u}$.

Now let $S$ be a $T_s$-set of sequences in $G$. Then $\tau_S$ is Hausdroff, and by its definition, it coincides with $\inf_{\mathbf{u}\in S} \tau_\mathbf{u}$.
Conversely, if $\inf_{\mathbf{u}\in S} \tau_\mathbf{u}$ is Hausdorff, then obviously $S$ is a $T_s$-set and $\tau_S =\inf_{\mathbf{u}\in S} \tau_\mathbf{u}$.
$\Box$

Let $G$ be an infinite group and $S$ be a set of sequences in $G$. The main result of this section is Theorem \ref{t11}. This theorem  allows us to check whether $S$ is a $T_s$-set of sequences, and it gives an explicit description of a base of the topology $\tau_S$ assuming that $S\in \mathcal{TS}(G)$. Such a description was given only for Abelian case of one $T$-sequence in \cite{ZP1, ZP2} (see the case (iii) below). Let us note that Theorem \ref{t11} generalizes Theorems 2.1.3 and 3.1.4 of \cite{ZP2} and it is used repeatedly in the article.

Let $G$ be a group and $S=\{ \mathbf{u}_i\}_{i\in I}, \mathbf{u}_i =\{ u^i_n\}_{ n\in\omega}$, be a non-empty family of sequences in $G$. By $\mathcal{J}$ we denote the set of all functions $\mathbf{j}$ from $\omega \times I\times G$ into $\omega$ which satisfy the condition
\[
\mathbf{j}(k,i,g) < \mathbf{j}(k+1,i,g), \quad \forall k\in \omega , \; \forall i\in I, \; \forall g\in G.
\]
 For $\mathbf{j}\in \mathcal{J}$ one puts:
\[
\begin{split}
 A_{m}^i & :=\{ e_G, \left(u_m^i\right)^{\pm 1}, \left(u_{m+1}^i \right)^{\pm 1}, \dots  \}, \; \forall m\in \omega;\\
A_n (\mathbf{j}) & := \bigcup_{i\in I} \bigcup_{g\in G} g^{-1} A^i_{\mathbf{j}(n,i,g)} g, \; \forall n\in \omega; \\
SP_{n\in \omega} A_{n} (\mathbf{j}) & := \bigcup_{ n\in\omega} \left( \bigcup_{\sigma \in \mathbb{S}_{n+1}} A_{\sigma(0)}(\mathbf{j}) A_{\sigma(1)}(\mathbf{j})\cdots A_{\sigma(n)}(\mathbf{j})\right) .
\end{split}
\]
Note that $SP_{n\in \omega} A_{n} (\mathbf{j})$ is an increasing union over $n\in\omega$.

The next lemma is \cite[Lemma 3.1.1]{ZP2}, we prove it for the convenience of the readers.
\begin{lemma} \label{l11}
Let $U$ be an open neighborhood of the unit in a topological group $G$. Then there is a decreasing chain $\{ V_k \}_{k\in\omega} \subset \mathcal{U}_G $ such that
\begin{equation} \label{222}
\bigcup_{ n\in\omega} \bigcup_{\sigma \in \mathbb{S}_{n+1}} V_{\sigma(0)} V_{\sigma(1)}\cdots V_{\sigma(n)} \subseteq U.
\end{equation}
\end{lemma}

\begin{proof}
Set $V_{-1} :=U$.  For every $k\in \omega$ choose a $V_k \in \mathcal{U}_G$ such that $(V_k \cup V_k V_k \cup V_k V_k V_k ) \subseteq V_{k-1}$.

Let $\{ i_0, i_1, \dots , i_n \}$ be a subset of $\omega$, and let $i_l =\min\{ i_0, i_1, \dots , i_n \}$. It is enough to show that $V_{i_0} V_{i_1} \cdots  V_{i_n} \subseteq V_{i_l -1}$.  The proof is by induction over $n\in\omega$. For $n=0$ we have $i_l =i_0$ and $V_{i_0} \subseteq V_{i_0 -1}$ by the choice of the sequence $\{ V_k \}$. Let $n>0$. If $0<l<n$, then, by the inductive assumption,
\[
(V_{i_0} V_{i_1} \cdots   V_{i_{l -1}}) V_{i_{l}} (V_{i_{l +1}}  \cdots  V_{i_n}) \subseteq V_{i_l } V_{i_l } V_{i_l } \subseteq V_{i_l -1}.
\]
If either $l=0$ or $l=n$, then $V_{i_0} V_{i_1} \cdots  V_{i_n} \subseteq V_{i_{l}} V_{i_{l}} \subseteq V_{i_{l} -1}$.
\end{proof}

\begin{theorem} \label{t11}
Let $G$ be a  group and $S=\{ \mathbf{u}_i\}_{i\in I}, \mathbf{u}_i =\{ u^i_n\}_{ n\in\omega}$, be a non-empty family of sequences in $G$. The following statements are equivalent:
\begin{enumerate}
\item[$1)$] $S\in \mathcal{TS}(G)$;
\item[$2)$] $\bigcap_{\mathbf{j}\in \mathcal{J}} SP_{n\in \omega} A_{n} (\mathbf{j}) = \{ e_G \}$.
\end{enumerate}
If $1)-2)$ are fulfilled, the sets $SP_{n\in \omega} A_{n} (\mathbf{j}), \mathbf{j}\in \mathcal{J},$  form an open base at the unit of $\tau_S$.
\end{theorem}

\begin{proof}
$1)\Rightarrow 2)$ Let $U\in \mathcal{U}_{(G,\tau_S)}$. Since $\tau_S$ is Hausdorff, it is enough to show that there is $\mathbf{j}\in \mathcal{J}$ such that $SP_{n\in \omega} A_{n} (\mathbf{j}) \subseteq U$. Moreover, we shall prove that for any Hausdorff group topology $\tau'$ on $G$ in which every $\mathbf{u}\in S$ converges to $e_G$ and for every $U\in \mathcal{U}_{(G,\tau')}$ there is $\mathbf{j}\in \mathcal{J}$ such that $SP_{n\in \omega} A_{n} (\mathbf{j}) \subseteq U$. We construct such a $\mathbf{j}$ inductively. Choose a sequence $\{ V_k \}_{k\in \omega}$ of open symmetric neighborhoods of $e_G$ which satisfies condition (\ref{222}) of Lemma \ref{l11}.

Let $k=0$. For every $i\in I$ and every $g\in G$, since $u^i_n \to e_G$ in $\tau'$, we may choose $\mathbf{j}(0,i,g)$ such that $g^{-1} u^i_n g \in V_0$ for every $n\geq \mathbf{j}(0,i,g)$. Then
\[
A_0 (\mathbf{j})= \bigcup_{i\in I} \bigcup_{g\in G} g^{-1} A^i_{\mathbf{j}(0,i,g)} g \subseteq V_0.
\]

Let $k=1$. For every $i\in I$ and every $g\in G$, since $u^i_n \to e_G$ in $\tau'$, we may choose $\mathbf{j}(1,i,g) > \mathbf{j}(0,i,g)$ such that $g^{-1} u^i_n g \in V_1$ for every $n\geq \mathbf{j}(1,i,g)$. Then
\[
A_1 (\mathbf{j})= \bigcup_{i\in I} \bigcup_{g\in G} g^{-1} A^i_{\mathbf{j}(1,i,g)} g \subseteq V_1.
\]
And so on. Thus we can construct $\mathbf{j}\in \mathcal{J}$ such that, by (\ref{222}),
\[
SP_{n\in \omega} A_{n} (\mathbf{j})  = \bigcup_{ n\in\omega} \bigcup_{\sigma \in \mathbb{S}_{n+1}} A_{\sigma(0)}(\mathbf{j}) A_{\sigma(1)}(\mathbf{j})\cdots A_{\sigma(n)}(\mathbf{j}) 
   \subseteq \bigcup_{ n\in\omega} \bigcup_{\sigma \in \mathbb{S}_{n+1}} V_{\sigma(0)} V_{\sigma(1)}\cdots V_{\sigma(n)}  \subseteq U.
\]

$2)\Rightarrow 1)$ The family $\mathcal{U}$ of all sets of the form $SP_{n\in \omega} A_{n} (\mathbf{j}), \mathbf{j}\in \mathcal{J},$ forms a base of some Hausdorff group topology on $G$ if and only if $\mathcal{U}$ satisfies the following conditions \cite[Theorem 4.5]{HR1}:
\begin{enumerate}
\item[{\rm (a)}] every member of $\mathcal{U}$ contains $e_G$;
\item[{\rm (b)}] for every $U\in \mathcal{U}$ there is a $V\in \mathcal{U}$ such that $V^{-1} \subseteq U$;
\item[{\rm (c)}] for every $U\in \mathcal{U}$ there is a $V\in \mathcal{U}$ such that $V\cdot V \subseteq U$;
\item[{\rm (d)}] for every $U\in \mathcal{U}$ and $g\in U$ there is a $V\in \mathcal{U}$ such that $g\cdot V\subseteq U$;
\item[{\rm (e)}]  for every $U\in \mathcal{U}$ and $h\in G$ there is a $V\in \mathcal{U}$ such that $h^{-1} V h \subseteq U$;
\item[{\rm (f)}] for every $U,V\in \mathcal{U}$ there is a $W\in \mathcal{U}$ such that $W\subseteq U\cap V$.
\end{enumerate}
Let us check (a)-(f) for $\mathcal{U}$. Fix $SP_{n\in \omega} A_{n} (\mathbf{j})\in \mathcal{U}$. (a) and (b) are obvious.

(c)  Set
\[
\mathbf{j}_1 (r,i,g) := \mathbf{j}(2r+1,i,g), \; \forall r\in\omega, \forall i\in I, \forall g\in G.
\]
It is clear that $\mathbf{j}_1 \in \mathcal{J}$ and $A_r (\mathbf{j}_1)=A_{2r+1} (\mathbf{j})$ for every $r\in \omega$. Let $k,l\in \omega$. Fix arbitrary $\alpha\in \mathbb{S}_{k+1}$ and $\beta \in \mathbb{S}_{l+1}$.

Assume that $k<l$. Put
\[
\begin{split}
\sigma (r)= 2\alpha(r) +1, & \mbox{ if } \; 0\leq r \leq k;\\
\sigma (k+1+q)= 2\beta(q) , & \mbox{ if } \; 0\leq \beta(q) \leq k \mbox{ and } 0\leq q \leq l;\\
\sigma (k+1+q)= k+1+\beta(q) , & \mbox{ if } \; k < \beta(q) \leq l \mbox{ and } 0\leq q \leq l.
\end{split}
\]
Then $\sigma\in \mathbb{S}_{k+l+1}$. Since $A_{m+1} (\mathbf{j}) \subseteq A_m (\mathbf{j})$ for $m\in \omega$, for every $0\leq r \leq k$ and $0\leq q \leq l$ we have
\[
\begin{split}
A_{\alpha(r)} (\mathbf{j}_1) =A_{2\alpha(r)+1} (\mathbf{j})=A_{\sigma (r)} (\mathbf{j}), & \mbox{ if } \; 0\leq r \leq k;\\
A_{\beta(q)} (\mathbf{j}_1) =A_{2\beta(q)+1} (\mathbf{j})\subseteq A_{2\beta(q)} (\mathbf{j})=A_{\sigma (k+1+q)} (\mathbf{j}), & \mbox{ if } \; 0\leq \beta(q) \leq k;\\
A_{\beta(q)} (\mathbf{j}_1) =A_{2\beta(q)+1} (\mathbf{j})\subseteq A_{k+1+\beta(q)} (\mathbf{j})=A_{\sigma (k+1+q)} (\mathbf{j}), & \mbox{ if } \; k < \beta(q) \leq l.
\end{split}
\]
Hence
\[
A_{\alpha(0)} (\mathbf{j}_1) \cdots A_{\alpha(k)} (\mathbf{j}_1)\cdot A_{\beta(0)} (\mathbf{j}_1)\cdots A_{\beta(l)} (\mathbf{j}_1) \subseteq A_{\sigma (0)} (\mathbf{j})\cdots A_{\sigma (k+l+1)} (\mathbf{j}) \subseteq SP_{n\in \omega} A_{n} (\mathbf{j}).
\]

Analogously we consider the cases $k>l$ and $k=l$. By the definition of the sets $SP_{n\in \omega} A_{n} (\mathbf{j})$, we obtain
\[
SP_{n\in \omega} A_{n} (\mathbf{j_1}) \cdot SP_{n\in \omega} A_{n} (\mathbf{j_1}) \subseteq SP_{n\in \omega} A_{n} (\mathbf{j}).
\]

(d) Let $x\in  SP_{n\in \omega} A_{n} (\mathbf{j})$ and $k$ be the smallest index $n$ such that $x\in A_{\alpha(0)}(\mathbf{j}) A_{\alpha(1)}(\mathbf{j})\cdots A_{\alpha(n)}(\mathbf{j})$ for some $\alpha\in \mathbb{S}_{n+1}$. Set
\[
\mathbf{j}_2 (n,i,g) := \mathbf{j}(k +1+n,i,g), \; \forall n\in\omega, \forall i\in I, \forall g\in G.
\]
Then $\mathbf{j}_2 \in \mathcal{J}$. Let $l\in \omega$ and $\beta\in \mathbb{S}_{l+1}$ be arbitrary. Set
\[
\begin{split}
\sigma (r)= \alpha(r), & \mbox{ if } \; 0\leq r \leq k;\\
\sigma (k+1+r)= k+1+ \beta(r) , & \mbox{ if } \; 0\leq r \leq l.
\end{split}
\]
Then $\sigma\in \mathbb{S}_{k+l+1}$.  For every $0\leq r\leq l$, we have $A_{\beta(r)}(\mathbf{j}_2)= A_{k+1+\beta(r)}(\mathbf{j})= A_{\sigma (k+1+r)}(\mathbf{j})$. Thus
\[
x\cdot A_{\beta(0)}(\mathbf{j}_2) A_{\beta(1)}(\mathbf{j}_2)\cdots A_{\beta(l)}(\mathbf{j}_2)\subseteq  A_{\sigma (0)} (\mathbf{j})\cdots A_{\sigma (k+l+1)} (\mathbf{j}) \subseteq SP_{n\in \omega} A_{n} (\mathbf{j}).
\]
Hence $x\cdot SP_{n\in \omega} A_{n} (\mathbf{j}_2) \subseteq SP_{n\in \omega} A_{n} (\mathbf{j})$.

(e)  Set $\mathbf{j}_3 (r,i,g) := \mathbf{j}(r,i,gh),  \forall r\in\omega, \forall i\in I, \forall g\in G$. It is clear that $\mathbf{j}_3 \in \mathcal{J}$ and  for every $r\in \omega$ and each $i\in I$ we have
\[
\bigcup_{g\in G} g^{-1} A^i_{\mathbf{j}_3 (r,i,g)}g  =\bigcup_{g\in G} g^{-1} A^i_{\mathbf{j}(r,i,gh)} g 
   = \bigcup_{g\in G} h(gh)^{-1} A^i_{\mathbf{j}(r,i,gh)} gh \cdot h^{-1} = h\left( \bigcup_{g\in G} g^{-1} A^i_{\mathbf{j}(r,i,g)}g\right) h^{-1}.
\]
Thus $A_{r}(\mathbf{j}_3) = h A_{r}(\mathbf{j})h^{-1}$ and $h^{-1} SP_{n\in \omega} A_{n} (\mathbf{j}_3) h = SP_{n\in \omega} A_{n} (\mathbf{j})$.

(f) Let $SP_{n\in \omega} A_{n} (\mathbf{j}_4), SP_{n\in \omega} A_{n} (\mathbf{j}_5)\in \mathcal{U}$. Set
\[
\mathbf{j}_6 (r,i,g) := \max\{ \mathbf{j}_4 (r,i,g), \mathbf{j}_5 (r,i,g)\}, \; \forall r\in\omega, \forall i\in I, \forall g\in G.
\]
Then $\mathbf{j}_6 (r,i,g) \in \mathcal{J}$ and $A^i_{\mathbf{j}_6 (r,i,g)} \subseteq A^i_{\mathbf{j}_4 (r,i,g)} \cap A^i_{\mathbf{j}_5 (r,i,g)}$. Hence
$SP_{n\in \omega} A_{n} (\mathbf{j}_6) \subseteq SP_{n\in \omega} A_{n} (\mathbf{j}_4)\cap SP_{n\in \omega} A_{n} (\mathbf{j}_5)$. Thus $\mathcal{U}$ is an open basis at $e_G$ for some Hausdorff group topology.

In the first part of the proof we have shown that the topology generated by $\mathcal{U}$ is finer than an arbitrary Hausdorff group topology $\tau$ on $G$ in which every sequence of $S$ converges to $e_G$. Thus, by definition, $\mathcal{U}$ is an open basis at the unit of $\tau_S$.
\end{proof}

Now we give the explicit form of the open basis at the unit in Theorem \ref{t11} both for the Abelian case and for the case of one $T$-sequence. These forms will be applied in the sequel
(see Proposition \ref{p11} and Theorem \ref{t15}).

(i) {\it The Abelian general case.}
Let $G$ be an Abelian group and $S=\{ \mathbf{u}_i\}_{i\in I}, \mathbf{u}_i =\{ u^i_n\}_{n\in\omega}$, be a non-empty family of sequences in $G$. By $\mathcal{M}$ we denote the set of all functions $\mathbf{m}$ from $\omega \times I$ into $\omega$ which satisfy the condition
\[
\mathbf{m}(k,i) < \mathbf{m}(k+1,i), \quad \forall i\in I, \; \forall k\in \omega .
\]
 For $\mathbf{m}\in \mathcal{M}$ one puts
\[
\begin{split}
A^i_m & := \left\{ 0, \pm u^i_n , \; n\geq m \right\}, \quad \forall m\in \omega ; \\
A_k (\mathbf{m}) & := \bigcup_{i\in I} A^i_{\mathbf{m}(k,i)}, \quad \forall k\in \omega ; \\
\sum_k A_k (\mathbf{m}) & := \bigcup_{ k\in \omega } \left( A_0 (\mathbf{m}) + A_1 (\mathbf{m})+\dots + A_k (\mathbf{m}) \right).
\end{split}
\]
Let $\mathbf{j}\in \mathcal{I}$. For every $k\in\omega$ and $i\in I$ set $\mathbf{m}(k,i)=\min\{ \mathbf{j}(k,i,g): g\in G \}$. Clearly, $\mathbf{m}\in \mathcal{M}$. Since $G$ is Abelian we have
\[
\bigcup_{g\in G} g^{-1} A^i_{\mathbf{j}(k,i,g)} g = A^i_{\mathbf{m}(k,i)} \mbox{ and hence } SP_{n\in \omega} A_{n} (\mathbf{j}) = \sum_k A_k (\mathbf{m}).
\]
Thus Theorem \ref{t11} asserts that $S\in \mathcal{TS}(G)$ if and only if $\bigcap_{\mathbf{m}\in \mathcal{M}} \sum_k A_k (\mathbf{m}) =\{ 0\}$. In such a case, the sets  $\sum_k A_k (\mathbf{m}), \mathbf{m}\in \mathcal{M},$ form an open base at zero of $\tau_S$.

(ii) {\it The non-Abelian case of one $T$-sequence.}
Let $\mathbf{u}=\{ u_n\}_{n\in \omega }$ be an arbitrary sequence in a group $G$. In this case the set of indices $I$ contains only one element. So, for the sake of simplicity we drop the index $i$ and denote by $\mathcal{L}=\mathcal{L} (G)$  the set of all functions $\mathbf{l}$ from $\omega \times G$ into $\omega$ which satisfy the condition
\[
\mathbf{l}(k,g) < \mathbf{l}(k+1,g), \quad \forall k\in \omega ,  \; \forall g\in G.
\]
 For $\mathbf{l}\in \mathcal{L}$ one puts:
\[
\begin{split}
 A_{m}^\mathbf{u} = A_m & :=\{ e_G, u_m^{\pm 1}, u_{m+1}^{\pm 1}, \dots  \}, \; \forall m\in \omega ;\\
A_n (\mathbf{l}) & :=  \bigcup_{g\in G} g^{-1} A_{\mathbf{l}(n,g)} g, \; \forall n\in \omega ; \\
SP_{n\in \omega} A_{n} (\mathbf{l}) & := \bigcup_{ n\in\omega} \bigcup_{\sigma \in \mathbb{S}_{n+1}} A_{\sigma(0)}(\mathbf{l}) A_{\sigma(1)}(\mathbf{l})\cdots A_{\sigma(n)}(\mathbf{l}) .
\end{split}
\]
Then $\mathbf{u}$ is a $T$-sequence if and only if $\bigcap_{\mathbf{l}\in \mathcal{L}} SP_{n\in \omega} A_{n} (\mathbf{l}) =\{ e_G \}$. In such a case, the sets  $SP_{n\in \omega} A_{n} (\mathbf{l}), \mathbf{l}\in \mathcal{L},$ form an open base at $e_G$ of $\tau_\mathbf{u}$.

Note also that the subgroup $\langle\mathbf{u}\rangle$ is open in $(G,\tau_\mathbf{u})$.

(iii) {\it The Abelian case of one $T$-sequence.}
This case was already described explicitly in Theorem 2.1.3 of \cite{ZP1}.
Let $\mathbf{u}=\{ u_n\}_{n\in \omega }$ be an arbitrary sequence in an Abelian group $G$. Following \cite{ZP1}, for every $m\in \omega $ and an increasing sequence $0\leq j_0 < j_1 <\dots$ one puts
\[
\begin{split}
A_{m}^\mathbf{u} = A_{m} & :=\{ 0, \pm u_n :\; n\geqslant m \};\\
\sum_n A_{j_n} & := \bigcup_{n\in \omega } \left( A_{j_0} +A_{j_1} +\dots +A_{j_n} \right).
\end{split}
\]
If $\mathbf{u}$ is a $T$-sequence, then the sets of the form $\sum_n A_{j_n} $ form a base of symmetric neighborhoods at zero of $\tau_\mathbf{u}$ \cite{ZP1, ZP2}. In what follows we shall use the following sets defined in \cite{ZP1}: for every $m,k \in \omega $  one puts
\[
A(k,m)  := \underset{k+1}{\underbrace{A_{m} +\dots + A_{m}}}.
\]
Clearly, the sets of the form $A_{m}$ and $A(k,m)$ are compact in $\tau_\mathbf{u}$.

\section{A characterization of $s$-groups} \label{sec2}

In the following proposition we characterize continuous homomorphisms from $(G,\tau_\mathbf{u})$ into a  topological group.
\begin{pro} \label{p11}
Let $\mathbf{u}$ be a $T$-sequence in a group $G$ and $p$ be an homomorphism from $(G,\tau_\mathbf{u})$ into a  topological  group $(X,\tau)$. Then $p$ is continuous if and only if $p(u_n)\to e_X$ in $\tau$.
\end{pro}

\begin{proof}
The necessity is evident. Let us prove the sufficiency.
Let $U\in \mathcal{U}_{X}$. By Lemma \ref{l11} choose a decreasing chain  $\{ V_k \}_{k\in \omega }$ of open symmetric neighborhoods of $e_X$  satisfying the condition (\ref{222}).

Let $k=0$. Since $p(u_n) \to e_X$, for every $g\in G$  we may choose $\mathbf{l}(0,g)$ such that $p(g^{-1} u_n g) \in V_0$ for every $n\geq \mathbf{l}(0,g)$. Since $V_0$ is symmetric, then
\[
p\left( A_0 (\mathbf{l})\right) = p\left( \bigcup_{g\in G} g^{-1} A_{\mathbf{l}(0,g)} g \right) \subseteq V_0.
\]

Let $k=1$. Since $p(u_n) \to e_X$, for every $g\in G$  we may choose $\mathbf{l}(1,g) > \mathbf{l}(0,g)$ such that $p(g^{-1} u_n g) \in V_1$ for every $n\geq \mathbf{l}(1,g)$. Then
\[
p\left( A_1 (\mathbf{l})\right) = p\left( \bigcup_{g\in G} g^{-1} A_{\mathbf{l}(1,g)} g \right) \subseteq V_1.
\]
And so on. Thus we constructed $\mathbf{l}\in \mathcal{L}$ such that, by Lemma \ref{l11},
\[
p\left( SP_{n\in \omega} A_{n} (\mathbf{l})\right)  = p\left( \bigcup_{ n\in\omega} \bigcup_{\sigma \in \mathbb{S}_{n+1}} A_{\sigma(0)}(\mathbf{l}) A_{\sigma(1)}(\mathbf{l})\cdots A_{\sigma(n)}(\mathbf{l}) \right) 
   \subseteq \bigcup_{ n\in\omega} \bigcup_{\sigma \in \mathbb{S}_{n+1}} V_{\sigma(0)} V_{\sigma(1)}\cdots V_{\sigma(n)}  \subseteq U.
\]
By the particular case (ii) of Theorem \ref{t11},  $SP_{n\in \omega} A_{n} (\mathbf{l})\in \mathcal{U}_{(G,\tau_\mathbf{u})}$. Thus $p$ is continuous.
\end{proof}

The following theorem generalizes  Proposition \ref{p11} and gives a convenient criterion for automatical continuity of homomorphisms from  $s$-groups.
\begin{theorem} \label{t12}
Let $G$ be a group, $S\in \mathcal{TS}(G)$ and $p$ be a homomorphism from $(G,\tau_S)$ into a topological group $(X,\sigma)$. The following statements are equivalent:
\begin{enumerate}
\item[{\rm (i)}] $p$ is continuous;
\item[{\rm (ii)}]  for   any $W\in \mathcal{U}_X$ the set $p^{-1} (W)$ is open in $\tau_\mathbf{u}$ for every $\mathbf{u}\in S$, i.e., $p^{-1} (W) \in \bigwedge_{\mathbf{u}\in S} \tau_\mathbf{u}$;
\item[{\rm (iii)}]  $p(u_n)\to e_X$ for every $\mathbf{u}=\{ u_n\} \in S$.
\end{enumerate}
\end{theorem}

\begin{proof}
(i) $\Rightarrow$ (iii) is trivial.

(iii) $\Rightarrow$ (ii). Let $p(u_n)\to e_X$ for every $\mathbf{u}=\{ u_n\} \in S$ and let $W\in \mathcal{U}_X$. By Proposition \ref{p11}, $p$ is a continuous homomorphism from $(G,\tau_\mathbf{u})$ into $X$ for every $\mathbf{u}\in S$. Thus  $p^{-1} (W)$ is open in $\tau_\mathbf{u}$ for every $\mathbf{u}\in S$.

(ii) $\Rightarrow$ (i).  We know that $\tau_S \subseteq \bigwedge_{\mathbf{u}\in S} \tau_\mathbf{u}$. So, assuming that $p$ is not continuous, a  family $\mathcal{U}_0 := \{ U\cap p^{-1} (W): U\in  \mathcal{U}_{(G,\tau_S)} \mbox{ and } W\in \mathcal{U}_X\}$ forms an open basis at $e_G$ of the Hausdorff group topology $\tau_0$ on $G$ that is strictly finer than $\tau_S$. To obtain a contradiction we have to show  that every sequence $\mathbf{u}\in S$ converges to the unit in $\tau_0$. By hypothesis, every set $U\cap p^{-1} (W)$ is open in $\tau_\mathbf{u}$. Since $\mathbf{u}$ converges to $e_G$ in $\tau_\mathbf{u}$,  $\mathbf{u}$ converges to the unit in $\tau_0$.
\end{proof}

The following proposition shows that we may restrict ourselves only to the sets of the form $S(G,\tau)$:
\begin{pro} \label{p02}
{\it Let $S$ be a $T_s$-set of sequences in a group $G$. Then $\tau_S =\tau_{S(G,\tau_S)}$. In particular, if $(G,\tau)$ is an $s$-group, then $\tau =\tau_{S(G,\tau)}$.}
\end{pro}

\begin{proof}
Since $S\subseteq S(G,\tau_S)$, we have $\tau_S \supseteq \tau_{S(G,\tau_S)}$ by the definition of the topology $\tau_S$. Let $id: (G,\tau_{S(G,\tau_S)}) \to (G,\tau_S), id(g)=g,$ be the identity map. For every $\mathbf{u}=\{ u_n\} \in S(G,\tau_S)$, by the definition of $S(G,\tau_S)$, $id(u_n )=u_n \to e_G$ in $\tau_S$. By Theorem \ref{t12}, $id$ is continuous. Thus $\tau_S \subseteq \tau_{S(G,\tau_S)}$ and hence $\tau_S = \tau_{S(G,\tau_S)}$.
\end{proof}

As a corollary of the last theorem we obtain the following characterization of topologies of the form $\tau_S$:
\begin{theorem} \label{t13}
Let $(G,\tau)$ be a Hausdorff  topological group and $S$ be a set of sequences in $G$. The following statements are equivalent:
\begin{enumerate}
\item[{\rm (i)}]  $S\in \mathcal{TS}(G)$ and $\tau =\tau_S$;
\item[{\rm (ii)}] for every homomorphism $p$ from $(G,\tau )$ into an arbitrary Hausdorff topological group $(X,\sigma)$, $p$ is continuous if and only if $p(u_n )\to e_X$ for each $\{ u_n \} \in S$.
\end{enumerate}
\end{theorem}

\begin{proof}
(i) $\Rightarrow$ (ii) follows from Theorem \ref{t12}.

(ii) $\Rightarrow$ (i). The identity map $id : (G,\tau)\to (G,\tau), id(g)=g,$ is continuous. Thus $u_n \to e_G$ in $\tau$ for each $\{ u_n \} \in S$. So $S\in S(G,\tau)$ and hence $S\in \mathcal{TS}(G)$. By the definition of the topology $\tau_S$ we have $\tau \subseteq \tau_S$. On the other hand, by hypothesis, the identity isomorphism  $id: (G,\tau)\to (G,\tau_S), id(g)=g,$ is continuous as well. So, $\tau \supseteq \tau_S$. Thus $\tau = \tau_S$.
\end{proof}

{\it Proof of Theorem} \ref{charac}.
Let $(G,\tau)$ be an $s$-group and $\tau =\tau_S$ for some $S\in \mathcal{TS}(G)$. Let $p: (G,\tau)\to (X,\sigma)$ be a sequentially continuous homomorphism. Then, in particular, $p(u_n )\to e_X$ for each $\{ u_n \} \in S$. Hence $p$ is continuous by Theorem \ref{t13}.

Conversely, let  every sequentially continuous homomorphism $p$ from $(G,\tau)$ into a Hausdorff topological group $(X,\sigma)$ be continuous. Setting $S=S(G,\tau)$ we obtain that $\tau =\tau_S$ by Theorem \ref{t13}.
$\Box$

\section{Basic properties of $s$-groups} \label{sec3}

We begin this section from the proof of Theorem \ref{t03}.

{\it Proof of Theorem} \ref{t03}.  Set $Q:=  \pi (S)$. We have to show that $G/X$ with the quotient topology is topologically isomorphic to $(G/X, \tau_Q)$. Since every sequence of $Q$ converges to the unit in the quotient topology $\tau$ on  $G/X$, $Q$ is a $T_s$-set of sequences in $G/X$ and $\tau$ is weaker than $\tau_Q$. If $\tau \not= \tau_Q$, we may find a strictly finer Hausdorff group topology $\tau'$ on $G$ in which every sequence $\mathbf{u}\in S$ converges to the unit, namely: the topology generated by the sets of the form $U\cap \pi^{-1} (W), U\in \tau_S, W\in \tau_Q$. This contradicts to the definition of the topology $\tau_S$.
$\Box$

\begin{lemma} \label{l21}
{\it Let $(G,\tau)$ and $(H,\nu)$ be  topological groups. Set $S=S(G,\tau)$, $R=S(H,\nu)$, $T= S(G\times H,\tau\times\nu)$ and
\[
S\times R :=  \left\{ \{ (u_n, v_n)\} :\; \mathbf{u}=\{ u_n \} \in S \mbox{ and } \mathbf{v}=\{ v_n \} \in R \right\}.
\]
Then $T=S\times R$. }
\end{lemma}
\begin{proof}
Denote by $\pi_G$ and $\pi_H$ the projections of $G\times H$ onto $G$ and $H$ respectively.
By the definition of the product topology we have: $(u_n, v_n)$ converges to the unit in $\tau_S \times\tau_R$ if and only if $\pi_G (u_n, v_n):=u_n \to e_G$ in $\tau_S$ and $\pi_H (u_n, v_n):=v_n \to e_H$ in $\tau_R$, i.e., if and only if $\mathbf{u}=\{ u_n \} \in S$  and $\mathbf{v}=\{ v_n \} \in R$.
\end{proof}

{\it Proof of Theorem} \ref{t04}. (i) $\Rightarrow$ (ii). Let $(G,\tau_S)$ and $(H,\tau_R)$ be $s$-groups, where, by  Proposition \ref{p02}, we may assume that $S=S(G,\tau_S)$ and $R= S(H,\tau_R)$. Set $T:= S(G\times H,\tau_S \times\tau_R)$. By Lemma \ref{l21}, $T=S\times R$.
 In particular, $\mathbf{u}= \{ u_n \}\in S$ if and only if $\{ (u_n, e_H)\} \in T$.

We claim that
\begin{equation} \label{21}
S\subseteq S(G,\tau_T|_G) \; \mbox{ and } \; R\subseteq S(H,\tau_T|_H).
\end{equation}
Let $\mathbf{u}= \{ u_n \} \in S$. Then $\{ (u_n, e_H)\} \in T$. So, for every open neighborhood $U$ of the unit in $\tau_T$ almost all members $(u_n, e_H)$ are contained in $U$ and, hence, in $U\cap (G\times \{ e_H\})$. Thus, $u_n \to e_G$ in $\tau_T |_G$ and $S\subseteq S(G,\tau_T|_G)$. Analogously, $R\subseteq S(H,\tau_T|_H).$

We claim that
\begin{equation} \label{22}
\tau_S = \tau_T|_G \; \mbox{ and } \; \tau_R = \tau_T|_H.
\end{equation}
Indeed, by the definition of $T$, $\tau_S \times\tau_R \subseteq \tau_T$. Thus, $\tau_S = \tau_S \times\tau_R
|_G \subseteq \tau_T|_G$. By the definition of the topology $\tau_S$ and (\ref{21}), we have $\tau_T|_G \subseteq \tau_S$. Thus, $\tau_T|_G =\tau_S$. Analogously, $\tau_T|_H =\tau_R$.

To prove that $\tau_S \times\tau_R = \tau_T$ we have to show that $\tau_S \times\tau_R \supseteq \tau_T$. Let $U$ be an arbitrary open neighborhood of the unit in $\tau_T$. Choose an open neighborhood of the unit $V$ in $\tau_T$ such that $V\cdot V \subseteq U$. For the natural projections $\pi_G$ and $\pi_H$ onto $G$ and $H$ respectively one puts
\[
V_G = \pi_G (V\cap (G\times \{ e_H\})) \mbox{ and } V_H = \pi_H (V\cap  (\{ e_G\}\times H)).
\]
By (\ref{22}), $V_G$ is open in $\tau_S$ and  $V_H$ is open in $\tau_R$. So $V_G \times V_H \in \tau_S \times \tau_R$ and
\[
V_G \times V_H \subseteq \left(V_G \times \{ e_H\}\right) \cdot \left(\{ e_G\} \times V_H \right) \subseteq V\cdot V \subseteq U.
\]
Thus, $\tau_S \times \tau_R =\tau_T$.

(ii) $\Rightarrow$ (i). By Theorem \ref{t03}, $G= \pi_G (G\times H)$ and $H= \pi_H (G\times H)$ are  $s$-groups.
$\Box$

\begin{lemma} \label{l22}
Let $(G,\tau)$ be a Hausdorff topological group. Then $S(G,\tau_{S(G,\tau)}) = S(G,\tau)$.
\end{lemma}

\begin{proof}
Since $\tau \subseteq \tau_{S(G,\tau)}$, then $S(G,\tau_{S(G,\tau)}) \subseteq S(G,\tau)$. Conversely, if $\mathbf{u}\in S(G,\tau)$, then $\mathbf{u}\in S(G,\tau_{S(G,\tau)})$ by the definition of $\tau_{ S(G,\tau)}$.
\end{proof}

We recall that the {\it sequential modification} $X^{seq}$ of a topological space $(X,\tau)$ is the set $X$ with the new topology $\tau^{seq}$ such that $U$ is open if and only if $U$ is sequentially open in $(X,\tau)$. In the case of Hausdorff  topological groups we can describe  $\tau^{seq}$ as follows:
\begin{pro} \label{p21}
{\it Let $(G,\tau)$ be a Hausdorff topological group. Then $\tau^{seq} = \bigwedge_{\mathbf{u}\in S(G,\tau)} \tau_\mathbf{u}$. }
\end{pro}

\begin{proof}
Let $U\in \bigwedge_{\mathbf{u}\in S(G,\tau)} \tau_\mathbf{u}$. We have to show that $U$ is sequentially open in $\tau$. Let $u_n \to g \in U$ in $\tau$. We may assume that $g=e_G$. Then $\mathbf{u}=\{ u_n\} \in S(G,\tau)$. Since $U\in \tau_\mathbf{u}$ and $\mathbf{u}$ converges to the unit in $ \tau_\mathbf{u}$ by definition, all but finitely many members of $\mathbf{u}$ are contained in $U$. Thus $U$ is sequentially open in $\tau$.

Conversely, let $U$ be sequentially open in $\tau$. Then $U$ is sequentially open in $ \tau_\mathbf{u}$ for every $\mathbf{u}\in S(G,\tau)$. (Indeed, if $v_n \to g\in U$ in $ \tau_\mathbf{u}$, then $v_n \to g\in U$ in $ \tau$ either. Because of $U$ is sequentially open in $\tau$, almost all $v_n$ are contained in $U$.) Since, by Theorem \ref{p01}, $(G,\tau_\mathbf{u})$ is sequential, $U$ is open in $\tau_\mathbf{u}$. Thus, $U\in \bigwedge_{\mathbf{u}\in S(G,\tau)} \tau_\mathbf{u}$.
\end{proof}

\begin{theorem} \label{t23}
Let $(G,\tau)$ be a Hausdorff  topological group.
\begin{enumerate}
\item[{\rm (i)}] A set $U$ is sequentially open in $\tau_{S(G,\tau)}$ if and only if $U$ is sequentially open in $\tau$, i.e., $\tau^{seq}_{S(G,\tau)}=\tau^{seq}$.
\item[{\rm (ii)}] The topology $\tau_{S(G,\tau)}$ is the finest Hausdorff group topology on $G$ whose open sets are sequentially open in  $\tau$.
\end{enumerate}
\end{theorem}

\begin{proof}
(i) Let $U$ be sequentially open in $\tau_{S(G,\tau)}$ and a sequence $\mathbf{g}=\{ g_n \}$ converge to $g\in U$ in $\tau$. We may assume that $g=e_G$. So   $\mathbf{g}\in S(G,\tau)$. By Lemma \ref{l22}, $\mathbf{g}\in S(G,\tau_{S(G,\tau)})$.  Since $U$ is sequentially open in $\tau_{S(G,\tau)}$,  $g_n \in U$ for all sufficiently large $n$. Hence $U$ is sequentially open in $\tau$.

Since $\tau\subseteq \tau_{S(G,\tau)}$, the converse assertion is trivial.

(ii) By (i) we have to show only the minimality of $\tau_{S(G,\tau)}$. Let $\tau_0$ be an arbitrary Hausdorff group topology on $G$ whose open sets are sequentially open in $\tau$. By the definition of $\tau_{S(G,\tau)}$, we have to prove that any $\mathbf{u}=\{ u_n\} \in S(G,\tau)$ converges to the unit also in $\tau_0$. Assume the converse and there is  an open neighborhood $U$ of  the unit in $\tau_0$ that does not contain an infinitely many terms $\{ u_{n_k} \}$ of $\mathbf{u}$. Set $\mathbf{v}=\{ u_{n_k} \}$. Then $\mathbf{v}\in S(G,\tau)$ and $\mathbf{v}\cap U=\varnothing$. Hence $U$ is not sequentially open in $\tau$. This is a contradiction.
\end{proof}

{\it Proof of Theorem} \ref{t05}. Since $(G,\tau)$ is sequential, by Theorem \ref{t23}(i), we have
\[
\tau_{S(G,\tau)} \supseteq \tau = \tau^{seq} = \tau_{S(G,\tau)}^{seq} \supseteq \tau_{S(G,\tau)}.
\]
Thus, $\tau =\tau_{S(G,\tau)} =\tau^{seq} =\bigwedge_{\mathbf{u}\in S(G,\tau)} \tau_\mathbf{u}$ by Proposition  \ref{p21}, and hence $(G,\tau)=(G, \tau_{S(G,\tau)})$ is an $s$-group.
$\Box$

\begin{theorem} \label{t24}
Let $(G,\tau)$ be a Hausdorff  topological group. The following statements are equivalent:
\begin{enumerate}
\item[{\rm (i)}] $(G,\tau_{S(G,\tau)})$ is sequential;
\item[{\rm (ii)}] $\tau_{S(G,\tau)}=\bigwedge_{\mathbf{u}\in S(G,\tau)} \tau_\mathbf{u}$;
\item[{\rm (iii)}] $\tau^{seq}$ is a Hausdorff group topology.
\end{enumerate}
\end{theorem}

\begin{proof}
(i)$\Rightarrow$(ii). Let $(G,\tau_{S(G,\tau)})$ be sequential. By the definition of sequential spaces, Theorem \ref{t23}(i) and Proposition \ref{p21}, we have $\tau_{S(G,\tau)} =\tau^{seq}_{S(G,\tau)}=\tau^{seq} =\bigwedge_{\mathbf{u}\in S(G,\tau)} \tau_\mathbf{u}$.

(ii)$\Rightarrow$(iii) follows from Proposition \ref{p21}.

(iii)$\Rightarrow$(i). By Theorem \ref{t23}(ii) and the hypothesis, $\tau_{S(G,\tau)}=\tau^{seq}$. Thus, by Theorem \ref{t23}(i), $\tau_{S(G,\tau)}= \tau^{seq} = \tau^{seq}_{S(G,\tau)}$. So $(G,\tau_{S(G,\tau)})$ is sequential.
\end{proof}

The following proposition is a simple observation that immediately follows from the definitions of $k$-spaces and $s$-groups.
\begin{pro} \label{pro1}
Every non-discrete Hausdorff topological group $(G,\tau)$ without infinite compact subsets is neither an $s$-group nor a $k$-space and $\mathbf{s}(G,\tau)= G_d$.
\end{pro}

\begin{proof}
Clearly, every convergent sequence in $(G,\tau)$ is trivial. Thus the topology $\tau_{S(G,\tau)}$ is discrete. Hence $\mathbf{s}(G,\tau)= G_d$ and $(G,\tau)$ is not an $s$-group.

Let us show that $(G,\tau)$ is not a $k$-space. Let $A$ be an arbitrary subset of $G$. Then for every compact subset $K$ of $(G,\tau)$ the intersection $A\cap K$ is finite and hence closed in $\tau$. Assuming that $(G,\tau)$ is a $k$-space we obtain that $A$ is closed in $\tau$. Since $A$ is arbitrary this means that $\tau$ is discrete that contradicts the assumption of the proposition. Thus $(G,\tau)$ is not a $k$-space.
\end{proof}

The next example is taken from \cite[Theorem 6]{BaT}. For the convenience of the readers we present a more transparent proof.
\begin{exa} \cite{BaT} \label{exa1}
{\it There is an Abelian countably infinite $s$-group $(G, \tau)$  containing a closed subgroup $\Delta$ such that $(\Delta, \tau|_\Delta)$ is  not discrete but contains no infinite compact subsets. In particular, $(\Delta, \tau|_\Delta)$ is  not an $s$-group.}
\end{exa}

\begin{proof}
Consider the metrizable topology $\tau'$ on $\mathbb{Z}_0^\mathbb{N}$ generated by the base $\{ U_n \}_{n\in\omega}$, where
\[
U_n = \{ (n_i)\in \mathbb{Z}_0^\mathbb{N} : \; n_i \in 2^n \cdot \mathbb{Z} \mbox{ for } i\geq 1\}, \quad n\in\omega .
\]
Set $G=(\mathbb{Z}_0^\mathbb{N}, \tau_\mathbf{e})\times (\mathbb{Z}_0^\mathbb{N}, \tau')$. By Theorems \ref{t04} and \ref{t05}, $G$ is an $s$-group.
Let $\Delta =\{ \left( (n_i), (n_i)\right): (n_i)\in \mathbb{Z}^\mathbb{N}_0 \}$ be the diagonal subgroup of $G$. Denote by $\tau_\Delta$ the induced topology of $\tau_\mathbf{e} \times \tau'$ on $\Delta$. We have to show that $(\Delta, \tau_\Delta)$ is neither an $s$-group nor a $k$-space. By Proposition \ref{pro1} it is enough to prove that $(\Delta, \tau|_\Delta)$ has no infinite compact subsets. We shall prove this in the following two steps.

{\it Step} 1. We claim that  $(\Delta, \tau_\Delta)$ is {\it not discrete}.

This follows from the fact that for every $n\in\omega$ and every open neighborhood $U=\sum_n A_{j_n} \in \tau_\mathbf{e}$ (see Section \ref{sec1}(iii)) of zero the intersection $U_n \cap U$ is infinite. Indeed, $U\cap U_n$ contains the sequence $\{ 2^n e_i\}_{i\geq m}$, where $m=j_{2^n}$.

{\it Step} 2. We claim that {\it every compact subset of $(\Delta, \tau_\Delta)$ is finite}.

Indeed, let $K$ be a compact subset of $(\Delta, \tau_\Delta)$. We shall identify $\Delta$ with $\mathbb{Z}_0^\mathbb{N}$. Clearly, $K$ is a compact subset in the topology  $\tau_\mathbf{e}$. Since $\mathbf{e}$ generates $\mathbb{Z}^\mathbb{N}_0$, by Theorem 4.1.4 of \cite{ZP2} (see also \cite[Lemma 2]{Ga1} or Theorem \ref{t32} below), $K$ is contained in $A(k,0)$ for some $k>0$. Assume for a contradiction that $K$ is infinite. Let $K=\{ a_n\}_{n\in\omega}$ be a one-to-one enumeration of the elements of $K$. Since all the coordinates of $a_n$ are less or equal to $k+1$, we obtain that $a_n -a_m \not\in U_{4k}$ for every $n\not= m$. So $\left( a_n + U_{8k}\right) \cap \left( a_m + U_{8k}\right) = \varnothing$ for every $n\not= m$ since, otherwise, $a_n -a_m \in U_{8k} -U_{8k} \subseteq U_{4k}$ that is impossible. Thus the sets $a_n + U_{8k}$ are open in $\tau_\Delta$, mutually disjoint  and  cover  $K$. So the set $K$ is not compact. This contradiction shows that $K$ is finite.
\end{proof}

Recall that the {\it supremum} $\tau :=\tau_1 \vee \tau_2$ of two group topologies $\tau_1$ and $\tau_2$  on a group $G$ is the weakest group topology on $G$ such that the identity maps $(G,\tau) \to (G,\tau_i), i = 1, 2$, are continuous. Clearly, the group $(G, \tau_1 \vee \tau_2)$ may be identified with the diagonal of the product $(G\times G, \tau_1 \times \tau_2)$.
Let us remind that the group $(\mathbb{Z}_0^\mathbb{N}, \tau_\mathbf{e})$ is sequential and that every sequential space is a $k$-space. So, as a corollary of Proposition \ref{pro1} and Example \ref{exa1}, we obtain:
\begin{cor} \label{cc1}
\begin{enumerate}
\item[{\rm (i)}] There is an Abelian countably infinite $s$-group $(G,\tau)$ with a closed subgroup $H$ such that $(H, \tau|_H)$ is not an $s$-group.
\item[{\rm (ii)}] There is an Abelian countably infinite $s$-group that is not a $k$-space.
\item[{\rm (iii)}] The supremum of two $s$-topologies may not be an $s$-topology.
\item[{\rm (iv)}] The product of a sequential group with a metrizable one may not be even a $k$-space.
\end{enumerate}
\end{cor}

\begin{rmk} {\em
It is well-known that the properties to be sequential or  Fr\'{e}chet-Urysohn are not well-behaved in general under the finite product, and   Corollary \ref{cc1}(iv) demonstrates this. Nevertheless,  the product of a first countable group with a Fr\'{e}chet-Urysohn one is always   Fr\'{e}chet-Urysohn  \cite[Theorem 1.6]{ChMPT}.}
\end{rmk}

\section{The case of countable $T_s$-sets of sequences} \label{sec4}

Let $(G,\tau)$ be an $s$-group. Then $\tau = \tau_S$ for some $S\in \mathcal{TS}(G)$.
Now we  discuss the minimality of $|S|$ of $T_s$-sets $S$ which generate $\tau$.
The following proposition shows that the number $r_s (G,\tau)$ is essentially infinite.
\begin{pro} \label{p12}
{\it Let $(G,\tau)$ be an $s$-group. If $r_s (G,\tau)$ is finite, then $r_s (G,\tau)=1$.}
\end{pro}

\begin{proof}
Let $S=\left\{ \mathbf{u}^0 =\{ u^0_n \}_{n\in\omega}, \dots , \mathbf{u}^{q-1} =\{ u^{q-1}_n \}_{n\in\omega} \right\}$ be such that $\tau =\tau_S$. Set $\mathbf{d} =\{ d_n \}$, where $d_{kq +i} =u^i_k$ for $k\in \omega$ and $0\leq i<q$. Since $S\in \mathcal{TS} (G)$, the sequence $\mathbf{d}$ converges to the unit in $\tau_S$. Hence $\mathbf{d}$ is a $T$-sequence and the topology $\tau_\mathbf{d}$ is finer than $\tau_S$. On the other hand, since $\mathbf{d}$ converges to $e_G$, then all its subsequences $\mathbf{u}^i$ also converge to $e_G$ in $\tau_\mathbf{d}$. So $\tau_S$ is finer than  $\tau_\mathbf{d}$ and hence $\tau_\mathbf{d} = \tau_S =\tau$. Thus $r_s (G,\tau)=1$.
\end{proof}

Proposition \ref{p12} shows that the simplest case which may essentially differ from the case of one $T$-sequence is the case of a countably infinite $T_s$-set of sequences $S$.  Example \ref{ex11} below confirms this assertion. Now we show that many important properties (for example, sequentiality and completeness) are kept also for a countably infinite $S$. Theorem \ref{t14} below generalizes Theorem \ref{p01}.

Let $\{ (X_n,\tau_n )\}_{n\in\omega}$ be a sequence of topological spaces such that $X_n \subseteq X_{n+1}$ and $\tau_{n+1} |_{X_n} = \tau_n$ for all $n\in\omega$.  The union $X=\cup_{n\in\omega} X_n$ with the weak topology $\tau$ (i.e., $U\in\tau$ if and only if $U\cap X_n \in\tau_n$ for every $n\in\omega$) is called the {\it inductive limit} of the sequence $\{ (X_n,\tau_n )\}_{n\in\omega}$ and it is denoted by $(X,\tau)= \underset{\longrightarrow}{\lim} (X_n,\tau_n )$. Recall (see \cite{ZP2}) that a topological space is called a $k_\omega$-{\it space} if it is the inductive limit of an increasing sequence of its compact subsets. A topological group $(G,\tau)$ is called a $k_\omega$-{\it group} if its underlying topological space is a $k_\omega$-space. Let us recall also that
the sets of the form $V_U^l  = \{ (x,y)\in G\times G : x^{-1} y \in U\}$,
where $U\in \mathcal{N}_G$, form a base of the left uniform structure on $(G,\tau)$.
In fact, the following theorem is contained in Chapter 4 of \cite{ZP2}. Our proof of the sequentiality of $(G,\tau_S)$ is essentially simpler than the one proposed in \cite{ZP2}.
\begin{theorem} \label{t14}
Let $S= \{ \mathbf{u}_n \}_{n\in\omega}$ be a countable $T_s$-set of sequences in a group $G$. Assume that $\langle \mathbf{u}_0, \mathbf{u}_1,\dots\rangle =G$. For every $n\in\omega$, put $K_n := \cup_{i=0}^n A^{\mathbf{u}_i}_0$, $X_{n} :=  K_n \cdots  K_n $ with $n+1$ factors $K_n$, and $\tau_n =\tau_S |_{X_n}$.
Then:
\begin{enumerate}
\item[{\rm 1)}] $K_n$ and $X_n$ are  compact metrizable subsets of $(G,\tau_S)$ for every $n\in\omega$;
\item[{\rm 2)}] $(G,\tau_S)=\underset{\longrightarrow}{\lim} (X_n,\tau_n)$. In particular, every compact subset of $(G,\tau_S)$ is contained in some $X_n$;
\item[{\rm 3)}] $(G,\tau_S)$ is a $k_\omega$-group;
\item[{\rm 4)}] $(G,\tau_S)$ is complete in the left uniformity;
\item[{\rm 5)}] if $\tau_S$ is not discrete, then $(G,\tau_S)$ is sequential and ${\rm so}(G,\tau_S)=\omega_1$.
\end{enumerate}
\end{theorem}

\begin{proof}
1) Since, by definition, every  $A^{\mathbf{u}_n}_0, n\in\omega$, is compact in $\tau_S$, all the sets $K_n$  and $X_n$ are  compact. Since every $K_n$  and $X_n$ are  countable, they are metrizable by \cite[Theorem 3.1.21]{Eng}.

2) By hypothesis, $G=\cup_{n\in\omega} X_n$. Since all the $(X_n,\tau_n)$ are compact, the inductive limit $(G,\tau'):=\underset{\longrightarrow}{\lim} (X_n,\tau_n)$ is a (Hausdorff) $k_\omega$-space, and every its compact subset is contained in some $X_n$ \cite[Lemma 9.3]{Ste}.

Let $\tau_0$ be the finest Hausdorff group topology on $G$ such that $\tau_0 |_{X_n}  =\tau_n =\tau_S |_{X_n}$ for every $n\in\omega$. Then $\tau_0 \supseteq \tau_S$. Since every sequence $\mathbf{u}_n$ converge to $e_G$ in $X_n$ and hence in $\tau_0$, we have $\tau_0 \subseteq \tau_S$ by the definition of $\tau_S$. Thus $\tau_0 = \tau_S$. Hence, to prove that $(G,\tau_S)=(G,\tau')$ it is enough to show that $\tau'$ is a group topology. This follows from Lemma 4.1.3 of \cite{ZP2}, but we give here a much simpler proof repeating word-for-word the proof of Theorem 1 of \cite{MMO}.

To show that $(G,\tau')$ is a topological group, we must to prove that the map $f:(G,\tau')\times (G,\tau') \to (G,\tau')$ given by $f(x,y)=xy^{-1}$ is continuous.

Since $(G,\tau')$ is a $k_\omega$-space, $(G,\tau')\times (G,\tau')$ is also a $k_\omega$-space. Thus, to show that $f$ is continuous we only have to show that $f$ is continuous on all compact subsets of $(G,\tau')\times (G,\tau')$.

Let $K$ be a compact subset of $(G,\tau')\times (G,\tau')$. Then $K\subseteq K_1 \times K_1$, where $K_1$ is a compact subset of $(G,\tau')$. Since $(G,\tau')$ is  a $k_\omega$-space with decomposition $G=\cup_{n\in\omega} X_n$, we see that $K_1 \subseteq X_n$, for some $n$. Thus
\[
f(K) \subseteq f(K_1 \times K_1) \subseteq f(X_n \times X_n) \subseteq X_{2n+1}.
\]
Noting that $K$ is compact and $\tau' \supseteq \tau_S$, we see that $K$ has the same induced topology as a subset of $(G,\tau')\times (G,\tau')$ as it has as a subset of $(G,\tau_S )\times (G,\tau_S )$. Since $\tau' |_{X_{2n+1}} =\tau_S |_{X_{2n+1}}$ and $\tau_S$ is a group topology, $f: K\to X_{2n+1}$ is continuous. So $f$ is continuous on all compact subsets of $(G,\tau')\times (G,\tau')$. Hence $(G,\tau')$ is a topological group.

3) follows from item 2 and the definition of $k_\omega$-groups.

4) follows from 3) and from Theorem 4.1.6 of \cite{ZP2}.

5) $(G,\tau_S)$ is sequential by items 1)-3) and Lemma 1.5 of \cite{ChMPT}. Exercise 4.3.1 of \cite{ZP2} asserts that the sequential order of $(G,\tau_S)$ is $\omega_1$.
\end{proof}

Proposition \ref{p12} asserts that every {\it finite} $T_s$-set of sequences of a group $G$ can be replaced by a single $T$-sequence. The next example shows that in general this proposition  cannot be generalized to all {\it countably infinite} $T_s$-sets of sequences.

\begin{exa} \label{ex11} {\em
Let $\{ G_n\}_{n\in \omega}$ be a sequence of infinite Abelian groups and $\mathbf{u}_n $ be a non-trivial $T$-sequence in $G_n$ for every $n\in \omega$. Then $S=\{ \mathbf{u}_n \}_{n\in \omega}$ is a $T_s$-set of sequences in the direct sum $G=\bigoplus_{n\in\omega} G_n$ because every  $\mathbf{u}_n$ converges to zero in the product topology of $\tau_{\mathbf{u}_n}$. We claim that $r_s (G,\tau_S)= \aleph_0$. Indeed, assuming the converse, by Proposition \ref{p12}, we can find a sequence $\mathbf{v}$ such that $\tau_\mathbf{v} =\tau_S$.  By Theorem \ref{t14},  $\mathbf{v}$ is contained in $G_0 +...+ G_k$ for some $k$. Hence the subgroup $G_0 +...+ G_k$ is open in $\tau_\mathbf{v}$. Thus, for every $n>k$, the sequence $\mathbf{u}_n$ does not converge to zero. This is a contradiction.}
\end{exa}

\begin{rmk} {\em
Let $S\in \mathcal{TS}(G)$ for a group $G$. In general, the group $(G,\tau_S)$ may be not complete. Indeed, let  $G$ be a countable dense subgroup of a compact infinite metrizable group. Then $G$ is metrizable and non complete.  By Theorem \ref{t05}, $(G,\tau_S)$ is an $s$-group, but it is not complete. }
\end{rmk}

{\it Proof of Theorem} \ref{t001}. Let $r_s (G,\tau) <\omega_1$ and $S$ be a countable $T_s$-set of sequences in $G$ such that $\tau = \tau_S$. Then, by Theorem \ref{t14}, the open subgroup $H$ of $(G,\tau)$ generated by $S$   is sequential and ${\rm so }(H,\tau)=\omega_1$. Hence also $(G,\tau)$  is sequential and ${\rm so }(G,\tau)=\omega_1$.
$\Box$

In the rest of this section we deal with the case of one $T$-sequence. In spite of this simplest case was thoroughly studied in \cite{ZP2}, to the best of our knowledge these results are new.

In the next theorem we prove that the topologies of the form $\tau_\mathbf{u}$ are well behaved under the product.
\begin{theorem} \label{t15}
Let $\mathbf{u}=\{ u_n\}_{n\in\omega}$ and $\mathbf{v}=\{ v_n\}_{n\in\omega}$ be  $T$-sequences in groups $G$ and $H$ respectively. Set $\mathbf{d}=\{ d_n\}_{n\in\omega}$, where $d_{2n+1} =(u_n,e_H)$ and $d_{2n} =(e_G ,v_n)$. Then $\mathbf{d}$ is a $T$-sequence in $G\times H$ and $\tau_\mathbf{d} =\tau_\mathbf{u}\times \tau_\mathbf{v}$.
\end{theorem}

\begin{proof}
It is clear that $\mathbf{d}$ converges to the unit in $\tau_\mathbf{u}\times \tau_\mathbf{v}$. So $\mathbf{d}$ is a $T$-sequence in $G\times H$ and $\tau_\mathbf{u}\times \tau_\mathbf{v} \subseteq \tau_\mathbf{d}$. To prove that $\tau_\mathbf{u}\times \tau_\mathbf{v}=\tau_\mathbf{d}$ it is enough to show that every basic neighborhood $W=SP_{n\in \omega} A_{n} (\mathbf{l}), \mathbf{l} \in \mathcal{L} (G\times H),$ of the unit in $\tau_\mathbf{d}$ contains a set of the form $W_\mathbf{u}\times W_\mathbf{v}$, where $W_\mathbf{u} \in \tau_\mathbf{u}$ and $W_\mathbf{v} \in \tau_\mathbf{v}$.

Let $\mathbf{l}^1, \mathbf{l}^2 \in \mathcal{L} (G\times H)$. We say that $\mathbf{l}^1 \leq \mathbf{l}^2$ if $\mathbf{l}^1 (k, (g,h)) \leq \mathbf{l}^2 (k, (g,h))$ for every $(k, (g,h))\in \omega\times G\times H$. Clearly, if $\mathbf{l}^1 \leq \mathbf{l}^2$, then $SP_{n\in \omega} A_{n} (\mathbf{l}^2) \subseteq SP_{n\in \omega} A_{n} (\mathbf{l}^1)$. So we may assume that for every $k\in\omega, g\in G, h\in H$,
\[
\mathbf{l} (2k, (g,h)) = 2\mathbf{l}' (k, (g,h))+1 \mbox{ and } \mathbf{l} (2k+1, (g,h)) = 2\mathbf{l}'' (k, (g,h)),
\]
where $\mathbf{l}' , \mathbf{l}'' \in \mathcal{L} (G\times H)$. For every $k\in\omega, g\in G, h\in H$, set
\[
\mathbf{l}^\mathbf{u} (k,g):=\mathbf{l}' (k, (g,e_H )) \mbox{ and } \mathbf{l}^\mathbf{v} (k,h):=\mathbf{l}'' (k,(e_G, h)).
\]
Then $\mathbf{l}^\mathbf{u} \in \mathcal{L} (G)$ and $\mathbf{l}^\mathbf{v} \in \mathcal{L} (H)$, and
\[
\begin{split}
A^\mathbf{u}_{\mathbf{l}^\mathbf{u} (k,g)} \times \{  e_H\} & = \left\{ e_G, u_{\mathbf{l}^\mathbf{u} (k,g)}^{\pm 1}, \dots\right\} \times \{ e_H\} = \left\{ (e_G, e_H), \left( u_{\mathbf{l}^\mathbf{u} (k,g)}^{\pm 1},e_H \right), \dots\right\} \\
 & = \left\{ (e_G, e_H), d_{\mathbf{l} (2k,(g,e_H))}^{\pm 1}, \dots\right\} \subset A^\mathbf{d}_{\mathbf{l} (2k,(g,e_H))}, \\
\{  e_G\} \times A^\mathbf{v}_{\mathbf{l}^\mathbf{v} (k,h)} & = \{ e_G\} \times \left\{ e_H, v_{\mathbf{l}^\mathbf{v} (k,h)}^{\pm 1}, \dots\right\}  = \left\{ (e_G, e_H), \left( e_G, v_{\mathbf{l}^\mathbf{v} (k,h)}^{\pm 1} \right), \dots\right\} \\
 & = \left\{ (e_G, e_H), d_{\mathbf{l} (2k+1,(e_G, h))}^{\pm 1}, \dots\right\} \subset A^\mathbf{d}_{\mathbf{l} (2k+1,(e_G, h))}.
 \end{split}
 \]
 Thus
\[
A_k (\mathbf{l}^\mathbf{u})  \times \{  e_H\} \subset A_{2k} (\mathbf{l}) \mbox{ and } \{  e_G\} \times A_k (\mathbf{l}^\mathbf{v}) \subset A_{2k+1} (\mathbf{l}).
\]

For every $n\in\omega$ and $\sigma' , \sigma'' \in \mathbb{S}_{n+1}$ put
\[
\sigma(k)=2\sigma' (k) \mbox{ and } \sigma(n+1+k)=2\sigma'' (k) +1 ,\;  0\leq k\leq n.
\]
Then $\sigma \in \mathbb{S}_{2n+1}$ and
\[
\begin{split}
& \left( A_{\sigma'(0)}(\mathbf{l}^\mathbf{u}) \cdots A_{\sigma'(n)}(\mathbf{l}^\mathbf{u}) \right) \times \left( A_{\sigma''(0)}(\mathbf{l}^\mathbf{v}) \cdots A_{\sigma''(n)}(\mathbf{l}^\mathbf{v}) \right) \\
& = \left( A_{\sigma'(0)}(\mathbf{l}^\mathbf{u}) \times \{ e_H \} \right) \cdots  \left( A_{\sigma'(n)}(\mathbf{l}^\mathbf{u}) \times \{ e_H \} \right) \cdot \\
& \quad\cdot  \left( \{ e_G \}\times A_{\sigma''(0)}(\mathbf{l}^\mathbf{v})  \right) \cdots \left( \{ e_G \}\times A_{\sigma''(n)}(\mathbf{l}^\mathbf{v})  \right)
\subseteq A_{\sigma(0)}(\mathbf{l}) \cdots A_{\sigma(2n+1)}(\mathbf{l}).
\end{split}
\]
Set $W_\mathbf{u} = SP_{n\in \omega} A_{n} (\mathbf{l}^\mathbf{u})\in \tau_\mathbf{u}$ and $W_\mathbf{v} = SP_{n\in \omega} A_{n} (\mathbf{l}^\mathbf{v})\in \tau_\mathbf{v}$. Then
\[
\begin{split}
W_\mathbf{u}\times W_\mathbf{v} & = \bigcup_{ n\in\omega} \bigcup_{\sigma' , \sigma'' \in \mathbb{S}_{n+1}} \left( A_{\sigma'(0)}(\mathbf{l}^\mathbf{u}) \cdots A_{\sigma'(n)}(\mathbf{l}^\mathbf{u}) \right) \times \left( A_{\sigma''(0)}(\mathbf{l}^\mathbf{v}) \cdots A_{\sigma''(n)}(\mathbf{l}^\mathbf{v}) \right) \\
  & \subseteq \bigcup_{ n\in\omega} \bigcup_{\sigma  \in \mathbb{S}_{2n+1}} A_{\sigma(0)}(\mathbf{l}) \cdots A_{\sigma(2n+1)}(\mathbf{l}) =W.
\end{split}
\]
\end{proof}

In the next theorem we show that, in fact,  an arbitrary (respectively Abelian) group of the form $(\langle\mathbf{u}\rangle,\tau_\mathbf{u})$ is a quotient group of $(F,\tau_\mathbf{e})$ (respectively $(\mathbb{Z}^\mathbb{N}_0, \tau_\mathbf{e})$):
\begin{theorem} \label{t02}
Let $\mathbf{u}=\{ u_n \}$ be a $T$-sequence in a (respectively Abelian) group $G$ such that $\langle\mathbf{u}\rangle =G$. Then $(G,\tau_\mathbf{u})$ is a quotient  group of $(F,\tau_\mathbf{e})$ (respectively $(\mathbb{Z}^\mathbb{N}_0, \tau_\mathbf{e})$) under the  homomorphism
\[
\pi \left( e^{\varepsilon_1}_{i_1} e^{\varepsilon_2}_{i_2}\dots e^{\varepsilon_m}_{i_m} \right) =u^{\varepsilon_1}_{i_1} u^{\varepsilon_2}_{i_2}\dots u^{\varepsilon_m}_{i_m}, \; \mbox{ where } \varepsilon_j =\pm 1,
\]
\[
\left( \mbox{respectively } \; \pi ((n_1, n_2,\dots, n_m, 0,\dots))= n_1 u_1 +n_2 u_2 +\dots +n_m u_m \right).
\]
\end{theorem}

\begin{proof}
It is clear that $\pi$ is an algebraic epimorphism. Since $\pi (e_n)=u_n \to e_G$ in $\tau_\mathbf{u} $,  $\pi$ is continuous by Proposition \ref{p11}. By Theorem \ref{t03}, the quotient group $(F, \tau_\mathbf{e})/ \ker\pi$ is topologically isomorphic to $(G,\tau_\mathbf{u})$.
\end{proof}

The last theorem of the section  shows that any topology of the form $\tau_\mathbf{u}$ on Abelian groups can be characterized by the smallness of its compact sets  (condition 2(b)).

\begin{theorem} \label{t32}
Let $(G,\tau)$ be a Hausdorff Abelian topological group. Then the following statements are equivalent:
\begin{enumerate}
\item[{\rm 1.}] $\tau =\tau_\mathbf{u}$ for some $T$-sequence $\mathbf{u}$ in $G$.
\item[{\rm 2.}]
\begin{enumerate}
\item[{\rm (a)}] $(G,\tau)$ is a $k$-space;
\item[{\rm (b)}] there is a sequence $\mathbf{u} $  in $G$ converging to zero in $\tau$ such that for every compact subset $K$ in $(G,\tau)$ there are $n\in \mathbb{N}$ and $g_1,\dots, g_m \in G$ for which
    \[
    K\subseteq \bigcup_{i=1}^m \left( g_i + A(n,0)\right).
    \]
\end{enumerate}
\end{enumerate}
\end{theorem}

\begin{proof}
$1. \Rightarrow 2.$ Let $\tau =\tau_\mathbf{u}$ for some $T$-sequence $\mathbf{u}$. By Theorem \ref{p01}, $G$ is a $k$-space. Condition (b) holds by \cite[Lemma 2]{Ga1}.

$2. \Rightarrow 1.$ By condition (b),  $\mathbf{u}$ is a $T$-sequence. By the definition of  $\tau_\mathbf{u}$, we have $\tau\subseteq \tau_\mathbf{u}$. In particular, every  compact subset $K$ in $(G,\tau)$ is closed in $\tau_\mathbf{u}$. By condition (b) and since $A(n,0)$ is compact in $\tau_\mathbf{u}$, the set $K$ is compact in $\tau_\mathbf{u}$ either. Hence $\tau$ and $\tau_\mathbf{u}$ have the same compact sets.  Since $(G,\tau)$ and $(G, \mathbf{u})$ are $k$-spaces, the topologies $\tau$ and $\tau_\mathbf{u}$ coincide.
\end{proof}

\begin{rmk} \label{r11} {\em
The following example shows that we can not drop the requirement on $G$ to be a $k$-space. Let $(\mathbb{Z}, \tau^b)$ be the group of integers $\mathbb{Z}$ equipped with the Bohr topology $\tau^b$.  By Glicksberg's theorem \cite{Gli}, $\mathbb{Z}_d$ and $(\mathbb{Z}, \tau^b)$ have the same compact subsets. Thus, every compact subset of $(\mathbb{Z}, \tau^b)$ is finite and hence $(\mathbb{Z}, \tau^b)$ has no non-trivial convergent sequences. So condition (b) holds if and only if a $T$-sequence $\mathbf{u}$ is  trivial.  On the other hand, for every trivial $T$-sequence $\mathbf{u}$ the group $(\mathbb{Z},\tau_\mathbf{u})$ is discrete and infinite. Since $\tau^b$ is precompact we have $\tau_\mathbf{u} \not=  \tau^b$. }
\end{rmk}

\section{Structure of  $s$-groups} \label{sec5}

All statements of this section are proved only for $F(X)$. All analogous assertions for the Abelian case follow from Theorem \ref{t03} and the fact that $A(X)$ is a quotient of $F(X)$  (note also that they can be proved analogously to the non-Abelian case).

We start from the following:
\begin{pro} \label{p31}
Let $(X,e)$ be a sequential Tychonoff space with basepoint $e$ and the topology $\tau$.  Then the Graev free (respectively  Graev free Abelian) topological group $(F(X), \tau_F)$ (respectively $(A(X), \tau_A)$) is an $s$-group.
\end{pro}

\begin{proof}
We will proof the proposition only for $F(X)$ since in the Abelian case the proof is the same. As usual, $F_2 (X)$  stands for a subset of $F(X)$ formed by all words whose reduced length is less or equal to $2$ (recall that the {\it reduced length} of an element $g\in F(X)$ is the number of letters in the reduced word representing $g$). Set
\[
S= \left\{ \{ u_n \} \subset F_2 (X)\, : \; u_n \to e \mbox{ in } \tau_F \right\}.
\]
The set $S$ contains  trivial sequences and hence it is not empty. We will show that $\tau_F = \tau_S$. By the definition of $\tau_S$ we have $\tau_F \subseteq \tau_S$. In particular, $\tau =\tau_F |_X \subseteq \tau_S |_X$. To prove the converse inclusion, by the definition of $\tau_F$, it is enough to show that $\tau_S |_X =\tau$.

Assuming the converse we can find a closed subset $E$ of $X$  in $\tau_S$  such that $E$ is not closed in $\tau$. Since $X$ is sequential, there exists a sequence $\{ a_n \} \subseteq E$ that converges to $a\in X\setminus E$ in $\tau$ and hence in $\tau_F$. Thus the sequence $u_n := a_n \cdot a^{-1} \in F_2 (X)$ converges to $e$ in $\tau_F$. By the definition of $\tau_S$, $u_n \to e$ in $\tau_S$ either. Hence $a_n = u_n \cdot a$ converges to $a$ in $\tau_S$. Since $E$ is closed in $\tau_S$, we obtain that $a\in E$. This is a contradiction.
\end{proof}

In what follows we need the following notion:
\begin{defin}
{\it Let $(G,\tau)$ be a non-discrete $s$-group. The set
\[
S^\ast (G,\tau)=\{ \mathbf{u} =\{ u_n\}_{n\in\omega } \in S(G,\tau): \; \mathbf{u} \mbox{ is one-to-one and } u_n \not= e \mbox{ for every } n\in \omega\}
\]
is called {\em the star} of the $s$-group $G$.}
\end{defin}

\begin{pro} \label{p32}
Let $(G,\tau)$ be a non-discrete $s$-group. Then
\begin{enumerate}
\item[{\rm (i)}] all the elements of all the sequences of $S^\ast (G,\tau)$ generate the whole group $G$;
\item[{\rm (ii)}] if $\mathbf{u}\in S^\ast (G,\tau)$ and $g\in G$, then $g^{-1}\mathbf{u} g \in S^\ast (G,\tau)$;
\item[{\rm (iii)}] $\tau =\tau_{S^\ast(G,\tau)}$.
\end{enumerate}
\end{pro}

\begin{proof}
Since $(G,\tau)$ is non-discrete, $S(G,\tau)$ contains a non-trivial sequence. Let $\mathbf{u}=\{ u_n\}_{n\in\omega} \in S(G,\tau)$ be an arbitrary non-trivial sequence. We will show the following:
\begin{enumerate}
\item[$(\alpha)$] {\it there is a sequence $\mathbf{v} \in S^\ast(G,\tau)$ such that $\mathbf{v}$ and $\mathbf{u}$ converge to the unit (or diverge) simultaneously in  any group topology on $G$.}
\end{enumerate}
We will construct such a $\mathbf{v}$  as follows.

Let $u_0 \not= e$. Since $\mathbf{u}\to e$ in $\tau$, there is at most finite set of indices $I_1$ such that $i\geq 1$ and $u_i =u_0$ for every $i\in I_1$. Set $v_0 = u_0$. Denote by $\mathbf{u}^1 = \{ u_n^1 \}_{n\geq 1} $  the sequence $\mathbf{u}\setminus \left[ \{ u_0\} \bigcup\cup_{i\in I_1} \{u_i\} \right]$ with the natural enumeration. If $u_0 =e$, set $\mathbf{u}^1 = \{ u_n^1 \}_{n\geq 1} $, where $u_n^1 =u_n$. It is clear that $\mathbf{u}^1 \to e$ in $\tau$.

Consider the sequence $\mathbf{u}^1$. Let $u_1^1 \not= e$. Since $\mathbf{u}^1 \to e$ in $\tau$, there is at most finite set of indices $I_2$ such that $i\geq 2$ and $u_i =u_1^1$ for every $i\in I_2$. Set $v_1 = u_1^1$. Denote by $\mathbf{u}^2 = \{ u_n^2 \}_{n\geq 2} $  the sequence $\mathbf{u}^1 \setminus \left[ \{ u_1^1 \} \bigcup\cup_{i\in I_2} \{u_i^1\} \right]$ with the natural enumeration. If $u_1^1 =e$, set $\mathbf{u}^2 = \{ u_n^2 \}_{n\geq 2} $, where $u_n^2 =u_n^1$. It is clear that $\mathbf{u}^2 \to e$ in $\tau$.

And so on. Since $\mathbf{u}$ is not trivial, this process is infinite and we  can construct the sequence $\mathbf{v}$ (it is a subsequence of $\mathbf{u}$) such that it is one-to-one and does not contain the unit, i.e., $\mathbf{v}\in S^\ast(G,\tau)$. It is clear that we can obtain the sequence $\mathbf{u}$ from the sequence $\mathbf{v}$ by the following way: there is the subset $I\subseteq\omega$ of indices and the natural number $n_i$ for every $i\in I$ such that it is need to add to  $\mathbf{v}$ exactly $n_i$ terms $v_i$ and some set (maybe infinite) of the units.  By construction, $\mathbf{v}$ and $\mathbf{u}$ converge to the unit simultaneously in  any group topology on $G$.

(i) By ($\alpha$), the set $S^\ast(G,\tau)$ is non empty. Let $\mathbf{u}\in S^\ast(G,\tau)$ and  $g\not= e$ be an arbitrary element of $G$. If $g\in \mathbf{u}$, then $g\in \langle S^\ast(G,\tau)\rangle$. If $g\not\in \mathbf{u}$ we may add  $g$ to $\mathbf{u}$, and then this new sequence also belongs to $S^\ast(G,\tau)$. Thus, $g\in \langle S^\ast(G,\tau)\rangle$ and (i) follows.

(ii) is trivial.

(iii) Since $S^\ast(G,\tau) \subset S(G,\tau)$, we have $ \tau \subseteq \tau_{S^\ast(G,\tau)}$. Conversely, let $\mathbf{u}\in S(G,\tau)$. By the definition of $s$-topology, we have to show only that $\mathbf{u}\to e$ in $\tau_{S^\ast(G,\tau)}$ either. If $\mathbf{u}$ is trivial or $\mathbf{u}\in S^\ast(G,\tau)$, this is clear. Assume that $\mathbf{u}$ is nether trivial nor belongs to $S^\ast(G,\tau)$. Then $\mathbf{u}\to e$ in $\tau_{S^\ast(G,\tau)}$ by ($\alpha$).
\end{proof}

Now we recall the construction of the Fr\'{e}chet-Urysohn fan over an arbitrary set $Q$ of sequences.
Let $Q=\{ \mathbf{u}\}_{\mathbf{u}\in Q}$ be a non-empty set of one-to-one sequences in a set $\Omega$. The disjoint sum of these sequences with basepoint $e$ is denoted by $X_Q$, i.e.,
\[
X_Q = \bigoplus_{\mathbf{u}\in Q} \mathbf{u} \oplus \{ e\} = X_Q^0 \oplus \{ e \}.
\]
Let $\mathcal{B} (Q)$ be the set of all functions from $Q$ into $\omega$. The topology $\nu_Q$ on $X_Q$ is defined as follows: each point of $X_Q^0$ is isolated and the base at $e$ is formed by the sets of the form
\[
W(\beta) := \{ e \}\cup \bigcup_{\mathbf{u}\in Q} \left\{ u_{\beta (\mathbf{u})}, u_{\beta (\mathbf{u})+1},\dots \right\} , \quad \beta\in \mathcal{B} (Q).
\]
Then $X_Q$ is a Fr\'{e}chet-Urysohn  Tychonoff space.

\begin{pro} \label{p33}
Let $S^\ast(G,\tau)$ be the star of a non-discrete (respectively Abelian) $s$-group $(G,\tau)$ and let $X_{S^\ast(G,\tau)} =X_{S^\ast(G,\tau)}^0 \oplus \{ e \}$ be the Fr\'{e}chet-Urysohn fan over $S^\ast(G,\tau)$.  Then $(G,\tau)$ is a quotient of the Graev free (respectively Abelian) topological group $F(X_{S^\ast(G,\tau)})$ (respectively $A(X_{S^\ast(G,\tau)})$). The quotient map is sequence-covering.
\end{pro}
\begin{proof}
By Proposition \ref{p32}(iii) we may assume that $\tau = \tau_{S^\ast(G,\tau)}$.

Define the following map from $X_{S^\ast(G,\tau)}$ into $(G,\tau)$:
\[
p(x) = x, \mbox{ if } x\in X_{S^\ast(G,\tau)}^0, \mbox{ and } p(e)= e_G.
\]

We claim that $p$ is {\it  continuous.} Since only $e$ is a non-isolated point in $X_{S^\ast(G,\tau)}$, we have to show only the continuity of $p$ at $e$. Let $\mathbf{j}\in \mathcal{J}$ and $SP_{n\in \omega} A_{n} (\mathbf{j}) \in \mathcal{U}_G$. Set
\[
\beta(\mathbf{u}) := \mathbf{j} (0,\mathbf{u}, e_G) \mbox{ for every } \mathbf{u}  \in S^\ast(G,\tau).
\]
Then
\[
p(W(\beta)) \subseteq \bigcup_{\mathbf{u}\in S^\ast(G,\tau)} A^\mathbf{u}_{\mathbf{j} (0,\mathbf{u}, e_G)} \subset A_0 (\mathbf{j}) \subset SP_{n\in \omega} A_{n} (\mathbf{j}),
\]
and so $p$ is continuous.

By the definition of $F(X_{S^\ast(G,\tau)})$ we can extend $p$ to a continuous homomorphism  from $F(X_{S^\ast(G,\tau)})$ into $(G,\tau)$ that will be denoted also by $p$. By the construction of $X_{S^\ast(G,\tau)}$ and Proposition \ref{p32}(i), $p(X_{S^\ast(G,\tau)})$ generates $G$ and hence $p$ is surjective. Taking into consideration assertion ($\alpha$) in the proof of Proposition \ref{p32}, $p$ is sequence-covering.

We claim that $p$ is a {\it quotient map}.
To prove this we have to show that $p$ is open \cite[Theorem 5.27]{HR1}, i.e., the image of an arbitrary open neighborhood of $e$ in $F(X_{S^\ast(G,\tau)})$ contains a neighborhood of $e_G$ in $(G,\tau)$.

Let $U$ be an arbitrary open neighborhood of $e$ in $F(X_{S^\ast(G,\tau)})$. By Lemma \ref{l11}, there exists a decreasing sequence of symmetric neighborhoods $\{ V_n \}_{n\in \omega}$ of $e$ in $F(X_{S^\ast(G,\tau)})$ such that
\begin{equation} \label{1}
\bigcup_{ n\in\omega} \bigcup_{\sigma \in \mathbb{S}_{n+1}} V_{\sigma(0)} V_{\sigma(1)}\cdots V_{\sigma(n)} \subseteq U.
\end{equation}
Since $V_n \cap X_{S^\ast(G,\tau)}$ is an open neighborhood of $e$ in $X_{S^\ast(G,\tau)}$, we can find a sequence $\{ \beta_n\}_{n\in\omega} \in \mathcal{B}({S^\ast(G,\tau)})$ such that
\begin{enumerate}
\item $\beta_n (\mathbf{u}) < \beta_{n+1} (\mathbf{u})$ for all $\mathbf{u}\in S^\ast(G,\tau)$, and
\item $W(\beta_n ) \subseteq V_n$.
\end{enumerate}
By Proposition \ref{p32}(ii), we may define
\[
\mathbf{j} (n,\mathbf{u},g) := \beta_n (g^{-1} \mathbf{u}g), \; \forall n\in\omega, \forall \mathbf{u}\in S^\ast(G,\tau) , \forall g\in G.
\]
Then, by the choice of $\beta_n$, $\mathbf{j} \in \mathcal{J}$. For every $g\in G$ and each $\mathbf{u}\in S^\ast(G,\tau)$ we have
\[
g^{-1} \left\{ u_{\mathbf{j} (n,\mathbf{u},g) }, u_{\mathbf{j} (n,\mathbf{u},g) +1},\dots \right\} g =
\left\{ g^{-1} u_{\beta_n (g^{-1} \mathbf{u}g) } g, \dots \right\} \subset p(W(\beta_n)).
\]
Since $V_n$ is symmetric, we obtain that
\[
g^{-1} A^\mathbf{u}_{ \mathbf{j} (n,\mathbf{u},g) } g \subset p(V_n), \; \forall n\in\omega, \forall \mathbf{u}\in S^\ast(G,\tau), \forall g\in G.
\]
Hence $A_n (\mathbf{j}) \subseteq p(V_n)$ for every $n\in \omega$. By (\ref{1}), we have $SP_{n\in \omega} A_{n} (\mathbf{j}) \subseteq p(U)$. Thus $p$ is open.
\end{proof}

{\it Proof of Theorem} \ref{t06}. (i) $\Rightarrow$ (iii) follows from Proposition \ref{p33}.

(iii) $\Rightarrow$ (iv) is trivial.

(iv) implies (i) by Proposition \ref{p31} and Theorem \ref{t03}.

(ii) $\Rightarrow$ (iv) is trivial.

(iv) $\Rightarrow$ (ii). It is enough to show that every Graev free topological group over a sequential space is a quotient group of the Graev free topological group over a metrizable space.
Let $Y$ be a sequential space and $e_Y$ a base point in $Y$. By Franklin's theorem \ref{t01}, $Y$ is a quotient space of a metric space $X$. Let $q:X\to Y$ be the quotient mapping. Choose
an arbitrary preimage $e_X$ of $e_Y$ in $X$, that is, $q(e_X)=e_Y$. Then $F(Y)$ is  a quotient group of $F(X)=F(X,e_X)$ by Corollary 7.1.9 of \cite{ArT}. For the convenience of the reader we repeat this proof.

By the definition of $F(X)$, $q$ extends to a continuous homomorphism $\hat{q}: F(X)\to F(Y)$. Since $q$ is onto $\hat{q}$ is onto as well. We have to show only that $\hat{q}$ is open.

Denote by $\mathcal{T}_q$ the family of all images $\hat{q}(U)$, where $U$ is open in $F(X)$. Since $\hat{q}$ is an epimorphism, $\mathcal{T}_q$ is a group topology on the abstract group $F_a (Y)$ (i.e., $F_a (Y)$ is the underlying abstract group $F(Y)$). Since $\hat{q}$ is continuous $\mathcal{T}_q$ is not weaker  than the topology $\mathcal{T}$ of $F(Y)$, i.e., $\mathcal{T}\leq \mathcal{T}_q$.

Let us show that $\mathcal{T}_q$ induces on $Y$ its original topology $\tau_Y$. It suffices to verify that the set $V=\hat{q}(U)\cap Y$ is open in $Y$ for every open subset $U$ of $F(X)$. Denote by
$N$ the kernel of $\hat{q}$. It is easy to see that $q^{-1} (V)= X\cap NU$. Thus $q^{-1} (V)$ is open in $X$. Since $q$ is quotient, the set $V$ is open in $Y$.

Since $\mathcal{T}$ is the finest group topology on $F_a(Y)$ that induced on $Y$ the topology $\tau_Y$ \cite{Gra}, we obtain that $\mathcal{T}_q =\mathcal{T}$. In other words, $\hat{q}(U)$ is open in $F(Y)$ for every open set $U$ in $F(X)$. Thus $\hat{q}$ is a quotient map.
$\Box$

\begin{rmk} {\em
Let us note that we cannot weaken Theorem \ref{t06} to spaces of countable tightness. Indeed, let $G$ be a countable group without convergent sequences (for example, the group of integers $(\mathbb{Z}, \tau^b)$ with the Bohr topology). Clearly, the groups $G$ and $F(G)$ have countable tightness. By Lemma 2.11 of \cite{Mor}, $G$ is a quotient group of $F(G)$. Since $G$ is not an $s$-group, $F(G)$ is not an $s$-group as well by Theorem \ref{t03}. }
\end{rmk}

\section{Open questions}

Since every sequential group is an $s$-group, it is interesting to characterize the class $\mathbf{Seq}$ of all sequential groups as a subclass of the class
$\mathbf{S}$ of all $s$-groups.
\begin{problem}
Characterize the class $\mathbf{Seq}$ as a subclass of the class $\mathbf{S}$.
\end{problem}
Recall that a space $X$ is of {\it countable tightness} if for each set $A\subseteq X$, and every $x$ from the closure $\mathrm{cl}(A)$ of $A$, there exists a countable subset $B\subseteq A$ such that $x\in \mathrm{cl}(B)$. Recall also that every sequential space has countable tightness. It is well-known  that every closed subgroup of a sequential group is sequential and hence an $s$-group.
\begin{problem}  \label{prob1}
Let $(G,\tau)$ be an $s$-group of countable tightness such that all its closed subgroups are also  $s$-groups. Is $(G,\tau)$ sequential?
\end{problem}
Since every sequential Hausdorff group must be also a $k$-space, one can ask:
\begin{problem}  \label{prob2}
Let $(G,\tau)$ be an $s$-group  of countable tightness that is also a $k$-space. Must $(G,\tau)$ be sequential?
\end{problem}
Let us note that the condition on $(G,\tau)$ to have countable tightness in Problems \ref{prob1} and \ref{prob2} is essential. Indeed, it is known (see \cite{Arh}) that every  locally compact group  of countable tightness  is metrizable. Let $G$ be the direct product of $\kappa$-many copies of the torus $\mathbb{T}$ where $\aleph_0 < \kappa < \mathfrak{s}$. Then, by Theorem \ref{Nob}, $G$ is a compact (hence a $k$-space) non-metrizable $s$-group. It can be shown that every subgroup of $G$ is an $s$-group as well. Thus the answers to these problems without the assumption of countable tightness are negative.

We would like to know which groups from the most important classes of topological groups are contained in $\mathbf{S}$.
\begin{problem} \label{prob3}
Which locally compact  groups are $s$-groups?
\end{problem}
It is well-known that every Abelian locally compact group $(G,\tau)$ is topologically isomorphic to the direct product $\mathbb{R}^n \times H$, where $n\in\omega$ and $H$ has an open compact subgroup $K$. Thus, by Theorem \ref{t04}, $(G,\tau)$ is an $s$-group if and only if $K$ is an $s$-group. So in the  Abelian case we ask:
\begin{problem} \label{prob4}
Which Abelian compact  groups are $s$-groups?
\end{problem}

The case of vector spaces is of independent interest. For a topological space $X$ let $C_p (X)$ (respectively $C_c (X)$) be the set of all real-valued continuous function with the pointwise (respectively compact-open) topology.
\begin{problem} \label{prob5}
Which real vector spaces (for example, $C_p (X)$ or $C_c (X)$) are $s$-groups?
\end{problem}
Let us recall that by result of Pytkeev \cite{Pyt} and Gerlits \cite{Ger}, if a vector space $V$ is either $C_p (X)$ or $C_c (X)$, then $V$ is Fr\'{e}chet-Urysohn $\Leftrightarrow$ it is sequential
$\Leftrightarrow$ it is a $k$-space. It is interesting to know whether this sequence of equivalences can be extended:
\begin{problem} \label{prob6}
Let $V=C_p (X)$ or $V=C_c (X)$ be an $s$-group. Is $V$ Fr\'{e}chet-Urysohn?
\end{problem}

A cover $\mathcal{U}$ of a space $X$ is called an {\it $\omega$-cover} if every finite subset of $X$ is contained in some element $U\in \mathcal{U}$. A space $X$ has {\it property} $(\gamma)$ if every open $\omega$-cover $\mathcal{U}$ of $X$ contains a sequence $\{ U_n \}_{n\in\omega} \subseteq \mathcal{U}$ such that $X= \bigcup_{n\in\omega} \bigcap_{k\geq n} U_k$. It is known that  $C_p (X)$ is sequential if and only if $X$ has property $(\gamma)$ \cite{GeN}. Thus, if for some $X$ that has no property $(\gamma)$  the space  $C_p (X)$ is an $s$-group, we disprove Problem \ref{prob6} for  $C_p (X)$. So, prove or disprove  Problem \ref{prob6} we can answering to the following question:
\begin{problem} \label{prob7}
For which topological spaces $X$ the spaces  $C_p (X)$ and $C_c (X)$ are $s$-groups?
\end{problem}

By Corollary \ref{c01},  $r_s (G,\tau)\geq \omega_1$ for every Fr\'{e}chet-Urysohn group $(G,\tau)$.
\begin{problem} \label{prob8}
Is there a  (countable) Fr\'{e}chet-Urysohn  group $(G,\tau)$ such that $r_s (G,\tau)=\omega_1$?
\end{problem}
\begin{problem} \label{prob9}
Let $(G,\tau)$ be a separable  metrizable group. Is $r_s (G,\tau)= \mathfrak{c}$?
\end{problem}
Problems \ref{prob8} and \ref{prob9} are of independent interest by the following. Let us recall the Malyhin problem \cite[Problem 3.11]{Sha}: is every countable Fr\'{e}chet-Urysohn group metrizable? Now, assume for the moment that there is a countable Fr\'{e}chet-Urysohn   group $(G,\tau)$ such that $r_s (G,\tau)=\omega_1$ and the answer to Problem  \ref{prob9} is in the affirmative. Then the group  $(G,\tau)$ disproves the Malyhin problem under the negation of the Continuum Hypothesis. This strengthens Malyhin's result \cite[Corollary 2.10]{Sha}: $\mathrm{MA}+ \rceil\mathrm{CH}$ implies the existence of a countable
Fr\'{e}chet-Urysohn  group that is not metrizable.

By Proposition \ref{p31}, if $X$ is a sequential space with a fixed point, then $F(X)$ is an $s$-group. It  remains open whether the converse is true:
\begin{problem}
{\rm (A. Leiderman)} Let the Graev free topological group $F(X)$ over a Tychonoff space $X$ with a fixed point is an $s$-group. Is $X$ sequential?
\end{problem}

\noindent {\sc Acknowledgement:} It is a pleasure to thank the anonymous referee for careful reading the paper and finding inaccuracies and numerous suggestions which essentially improve the exposition of the article.

\end{document}